\documentclass[a4paper]{article}
\usepackage[utf8]{inputenc}
\usepackage[a4paper,left=4cm,right=4cm,top=3cm, bottom=3.5cm]{geometry}

\usepackage{mathtools}

\usepackage{amssymb}
\usepackage{amsthm} 
\usepackage{mathrsfs} 
\usepackage{dsfont} 
\usepackage{mathdots} 
\usepackage{tikz-cd} 
\usepackage{slashed} 

\usepackage[
    pdfusetitle,
    hidelinks
]{hyperref}
\usepackage[
    nameinlink,
    capitalise
]{cleveref}

\theoremstyle{plain}
\newtheorem{theorem}{Theorem}[section]
\newtheorem*{theorem*}{Theorem}
\newtheorem{corollary}[theorem]{Corollary}
\newtheorem{lemma}[theorem]{Lemma}

\theoremstyle{definition}
\newtheorem{definition}[theorem]{Definition}
\newtheorem{remark}[theorem]{Remark}
\newtheorem{example}[theorem]{Example}
\newtheorem*{example*}{Example}

\DeclarePairedDelimiterX{\set}[1]{\lbrace}{\rbrace}{\,#1\,}
\DeclarePairedDelimiterX{\setcond}[2]{\lbrace}{\rbrace}{\,#1\,\mathclose{}\delimsize :\mathopen{}\,#2\,}

\DeclarePairedDelimiter{\braket}{\langle}{\rangle}
\DeclarePairedDelimiter{\norm}{\lVert}{\rVert}
\DeclarePairedDelimiter{\abs}{\lvert}{\rvert}

\DeclareMathOperator{\Hom}{Hom} 
\DeclareMathOperator{\End}{End} 
\DeclareMathOperator{\Ext}{Ext}
\DeclareMathOperator{\Diff}{Diff} 
\DeclareMathOperator{\ind}{ind} 
\DeclareMathOperator{\dom}{Dom} 
\DeclareMathOperator{\ran}{Ran}
\DeclareMathOperator{\Ker}{Ker} 
\DeclareMathOperator{\supp}{supp} 
\DeclareMathOperator{\imag}{Im} 
\DeclareMathOperator{\Span}{Span} 
\DeclareMathOperator{\Op}{Op} 
\DeclareMathOperator{\diag}{diag} 
\DeclareMathOperator{\Res}{Res} 

\numberwithin{equation}{section}

\title{Relative \texorpdfstring{\(K\)}{K}-homology of higher-order differential operators}
\author{Magnus Fries \\ 
\href{mailto:magnus.fries@math.lth.se}{\nolinkurl{magnus.fries@math.lth.se} } \\ 
\href{https://orcid.org/0009-0001-4197-4501}{ORCID: 0009-0001-4197-4501} \\ 
Centre for Mathematical Sciences, Lund University, Sweden}

\date{\today}

\begin{document}

\maketitle

\begin{abstract}

We extend the notion of a spectral triple to that of a higher-order relative spectral triple, which accommodates several types of hypoelliptic differential operators on manifolds with boundary. The bounded transform of a higher-order relative spectral triple gives rise to a relative \(K\)-homology cycle. In the case of an elliptic differential operator on a compact smooth manifold with boundary, we calculate the \(K\)-homology boundary map of the constructed relative \(K\)-homology cycle to obtain a generalization of the Baum-Douglas-Taylor index theorem.

{
~\newline \noindent 
\textbf{MSC Classification: } 
\href{https://mathscinet.ams.org/mathscinet/msc/msc2020.html?t=19K56}{Index theory [19K56]}, 
\href{https://mathscinet.ams.org/mathscinet/msc/msc2020.html?t=58J32}{Boundary value problems on manifolds [58J32]},
\href{https://mathscinet.ams.org/mathscinet/msc/msc2020.html?t=19K35}{Kasparov theory (\(KK\)-theory) [19K35]},
\href{https://mathscinet.ams.org/mathscinet/msc/msc2020.html?t=58B34}{Noncommutative geometry (à la Connes) [58B34]}, 
\href{https://mathscinet.ams.org/mathscinet/msc/msc2020.html?t=47G30}{Pseudodifferential operators [47G30]}.
}

\end{abstract}

\tableofcontents

\section{Introduction}

In this paper, we study the index theory of differential operators from the viewpoint of non-commutative geometry. The index theory of elliptic differential operators on closed manifolds is fully described by the seminal Atiyah-Singer index theorem \cite{Atiyah_Singer_1963}. For manifolds with boundary, several approaches can be found in the literature, such as Atiyah-Bott\cite{Atiyah_Bott_1964}, index theory in the Boutet de Monvel calculus \cite{Boutet_de_Monvel_1971}, relative index theory in \cite{Gromov_Lawson_1983}, \cite{Bandara_2022} and \cite{Bunke_1995},  Atiyah-Patodi-Singer \cite{Atiyah_Patodi_Singer_1975b, Atiyah_Patodi_Singer_1975c,Atiyah_Patodi_Singer_1976a} and the related works by Agranovich-Dynin. Our aim in this paper is to abstract methodology in the literature beyond ellipticity and restrictions to first-order operators.

Atiyah observed that the wrong functoriality of \(K\)-theory was used in the proof of the Atiyah-Singer index theorem, motivating the introduction of a notion of abstract elliptic operators in \cite{Atiyah_1970} as cycles for what is now called analytic \(K\)-homology -- the homology theory dual to \(K\)-theory. The ideas of Atiyah were later formalized in Kasparov's \(KK\)-theory for \ensuremath{C^*}-algebras \cite{Kasparov_1980}. Moreover, the Atiyah-Singer index theorem of an elliptic differential operator can be reduced to that of Dirac operators \cite{Baum_Douglas_1982, Baum_van_Erp_2016}. Dirac operators are first-order elliptic differential operators constructed from a \(\text{spin}^c\)~ structure of a manifold, providing the foundation for Baum-Douglas' geometric K-homology \cite{Baum_Douglas_1982} as well as forming the prototypical example of a \emph{spectral triple} in Connes' non-commutative geometry \cite{Connes_1994}. Considering its path of development, \(K\)-homology is the natural framework for index theory of elliptic differential operators which can be seen in for example \cite{Baum_van_Erp_2018, Baum_van_Erp_2016, Baum_van_Erp_2021, Baum_Douglas_1982, Baum_Douglas_1982a, Roe_1993, Haskell_Wahl_2009, Connes_1994, Cuntz_Meyer_Rosenberg_2007} and \cite{Higson_Roe_2000}. \(K\)-homology is still the natural tool for index theory beyond first-order elliptic differential operators and \(\text{spin}^c\)~ geometry, such as for higher-order operators and the hypoelliptic operator arising on Carnot geometries. For instance, in \cite{Baum_van_Erp_2014} they solve the index theorem for hypoelliptic operators on contact manifolds using geometric \(K\)-homology. The rigidity and ubiquity of Kasparov's analytic \(K\)-homology naturally places the study of index theory for higher-order operators in analytic \(K\)-homology, without any assumption on order nor the ordinary notion of ellipticity.

This paper will describe how to place higher-order differential operators on manifolds with boundary into \(K\)-homology. The starting point is the relative \(K\)-homology developed in \cite{Baum_Douglas_Taylor_1989}, where they study elliptic first-order operators on manifolds with boundary using their boundary map in \(K\)-homology. Relative \(K\)-homology provides the necessary tools for calculating the boundary map in \(K\)-homology. A result we will extend to higher-order differential operators is the Baum-Douglas-Taylor index theorem \cite{Baum_Douglas_Taylor_1989} which says that
\begin{equation}
    \label{eq:intro_BDT_index_theorem}
    \partial [\Omega] = [\partial \Omega]
\end{equation}
for a compact \(\text{spin}^c\)~ manifold \(\overline{\Omega}\) with boundary (see also \cite[Proposition 11.2.15]{Higson_Roe_2000}). Here \([\Omega]\) and \([\partial \Omega]\) are the fundamental classes in \(K^*(C_0(\Omega))\) and \(K^{*-1}(C(\partial \Omega))\), respectively. The fundamental classes are described by \(\text{spin}^c\)~ Dirac operators constructed from the \(\text{spin}^c\)~-structure. In particular, \eqref{eq:intro_BDT_index_theorem} is an obstruction to lift relative \(K\)-homology cycle constructed from \(\text{spin}^c\)~ Dirac operators to \(K^*(C(\overline{\Omega}))\).

In \cite{Forsyth_Goffeng_Mesland_Rennie_2019}, they rephrased the relative \(K\)-homology cycles in \cite{Baum_Douglas_Taylor_1989} in the language of non-commutative geometry by constructing relative spectral triples that model first-order operators on manifolds with boundary. A relative spectral triple \((\mathcal{J}\vartriangleleft \mathcal{A}, \mathcal{H}, \mathcal{D})\) for an ideal inclusion \(J \vartriangleleft A\) of \ensuremath{C^*}-algebras is a loosening the requirement of the operator \(\mathcal{D}\) in a spectral triple to be self-adjoint, instead requiring \(\mathcal{D}\) to be symmetric and that \(j\dom \mathcal{D}^*\subseteq \dom \mathcal{D}\) for all \(j\in\mathcal{J}\) where \(\mathcal{J}\subseteq \mathcal{A}\cap J\) is a \(*\)-algebra dense in \(J\). To allow spectral triples to be constructed from higher-order differential operators, one would need to loosen the requirement of bounded commutators between \(\mathcal{D}\) and \(\mathcal{A}\). The approach found in \cite[Appendix A]{Goffeng_Mesland_2015} (an idea previously present in literature, for example in \cite[Lemma 51]{Grensing_2012} and \cite[Appendix 5.1]{Wahl_2009}) amounts to requiring that \([\mathcal{D},a](1+\mathcal{D}^2)^{-\frac{1}{2}+\frac{1}{2m}}\) extends to an element in \(\mathbb{B}(\mathcal{H})\) for any \(a\in \mathcal{A}\) for some \(m>0\), as opposed to \([\mathcal{D},a]\). 

To obtain a generalization of the Baum-Douglas-Taylor index theorem \cite{Baum_Douglas_Taylor_1989} for higher-order differential operators, we combine the generalizations of spectral triples in \cite{Forsyth_Goffeng_Mesland_Rennie_2019} and \cite[Appendix A]{Goffeng_Mesland_2015} and define \emph{higher-order relative spectral triples}. For the \(K\)-homology boundary map to be feasible to compute, we begin by assuring ourselves that a higher-order relative spectral triple gives rise to a relative \(K\)-homology cycle.

\begin{theorem}[Relative \(K\)-homology cycle defined from a higher-order relative spectral triple]
    \label{thm:intro_horst_defines_kcykle}
    Let \((\mathcal{A}\vartriangleleft \mathcal{J}, \mathcal{H}, \mathcal{D})\) be a higher-order relative spectral triple for \(J\vartriangleleft A\) (see~\Cref{def:higher_order_relative_spectral_triple}) and let \(F\coloneqq \mathcal{D}(1+\mathcal{D}^*\mathcal{D})^{-\frac{1}{2}}\) be the bounded transform of \(\mathcal{D}\). Then 
    \begin{equation}
        \rho(j)(F-F^*),\rho(j)(1-F^2),[F,\rho(a)]\in \mathbb{K}(\mathcal{H})
    \end{equation}
    for any \(j\in J\) and \(a\in A\). In particular, \((\rho, \mathcal{H}, F)\) defines a relative \(K\)-homology cycle for \(J\vartriangleleft A\) and give a class in relative \(K\)-homology \(K^*(J\vartriangleleft A)\). We denote the class of \((\rho, \mathcal{H}, F)\) as \([\mathcal{D}]\in K^*(J\vartriangleleft A)\) when no confusion arises.
\end{theorem}

For the definition of higher-order relative spectral triples, see \Cref{def:higher_order_relative_spectral_triple}. \Cref{thm:intro_horst_defines_kcykle} can be found as \Cref{thm:horst_defines_kcykle} in the bulk of the text. To provide examples of higher-order relative spectral triples, we present a localization procedure of a (non-relative) higher-order spectral triple where the operator acts in some sense locally, as made precise in \eqref{eq:intro_abstract_local_condition}, later presented as \Cref{thm:localization_horst}.

\begin{theorem}
    \label{thm:intro_localization_horst}
    Let \((\mathcal{A},\mathcal{H},\mathcal{D})\) be a higher-order spectral triple for \(A\) such that \(A\) is unital or \([\mathcal{D}, a](1+\mathcal{D}^2)^{-\frac{1}{2}}\in \mathbb{K}(\mathcal{H})\) for all \(a\in \mathcal{A}\). If \(J\vartriangleleft A\) is an ideal and \(\mathcal{J}\subseteq J\cap \mathcal{A}\) a \(*\)-ideal such that such that for each \(j\in \mathcal{J}\) there is a \(k\in \mathcal{J}\) such that 
    \begin{equation}
        \label{eq:intro_abstract_local_condition}
        j = kj \text{ and } \mathcal{D} j = k \mathcal{D} j,
    \end{equation}
    then 
    \begin{equation}
        (\mathcal{J}\vartriangleleft (\mathcal{A}/\mathcal{J}^\perp), \overline{\mathcal{J} \mathcal{H}}, \overline{\mathcal{D}|_{\mathcal{J}\dom \mathcal{D}}})
    \end{equation}
    defines a higher-order relative spectral triple for \(J\vartriangleleft (A/J^\perp)\), where \(\mathcal{J}^\perp = J^\perp\cap \mathcal{A}\)
\end{theorem}

Examples of (non-relative) higher-order spectral triples can be constructed from differential operators that are in some sense elliptic. This does not have to be ellipticity in the classical sense, but for example also ellipticity in the Shubin calculus on \(\mathbb{R}^n\) \cite[Chapter IV]{Shubin_2001} or the Heisenberg calculus on a compact filtered manifold (also called a Carnot manifold) \cite{van_Erp_Yuncken_2019, Dave_Haller_2022}. This is discussed further in \Cref{sec:host_examples}. As differential operators act locally in the sense of \eqref{eq:intro_abstract_local_condition}, we also obtain examples of higher-order relative spectral triples.

\begin{corollary}[Geometric examples of higher-order relative spectral triples]
    \label{thm:intro_diff_op_horst}
    Let \(M\) be a smooth manifold and \(D\colon C^\infty(M; E)\to C^\infty(M; E)\) be formally self-adjoint differential operator that is elliptic in any pseudo-differential calculus presented in \Cref{sec:host_examples}, then for any open set \(\Omega\subseteq M\) the collection
    \begin{equation}
        (C_c^\infty(\Omega)\vartriangleleft C_c^\infty(\overline{\Omega}), L^2(\Omega;E), D_{\min})
    \end{equation}
    defines a higher-order relative spectral triple of order \(m\) for \(C_0(\Omega)\vartriangleleft C_0(\overline{\Omega})\) where \(\dom D_{\min}\) is the closure of \(C_c^\infty(\Omega;E)\) with the graph-norm of \(D\).
\end{corollary}

The \(K\)-homology class \([D_{\min}]\) is shown in \Cref{thm:independent_of_lot} to be independent of perturbations of low order and is hence only dependent on an appropriate notion of principal symbol of \(D\).

For a differential operator \(D\colon C^\infty(M; E)\to C^\infty(M; F)\) that is elliptic in an appropriate pseudo-differential calculus, we introduce the shortened notation 
\begin{equation}
    [D] \coloneqq  \left[\begin{pmatrix}0 & D^\dagger_{\min} \\ D_{\min} & 0\end{pmatrix}\right] \in K^0(C_0(\Omega)\vartriangleleft C_0(\overline{\Omega}))
\end{equation}
where \(D^\dagger\) denotes the formal adjoint of \(D\).

The \(K\)-homology class of a higher-order relative spectral triple \((\mathcal{A}\vartriangleleft \mathcal{J}, \mathcal{H}, \mathcal{D})\) for \(J\vartriangleleft A\) is independent of extension of \(\mathcal{D}\), in the sense that for any extension \(\mathcal{D}_e\) of \(\mathcal{D}\) such that \(\mathcal{D}\subseteq \mathcal{D}_e\subseteq \mathcal{D}^*\), the bounded transform of \(\mathcal{D}_e\) defines a \(K\)-homology cycle for \(J\) and
\begin{equation}
    \label{eq:intro_independence_of_extension}
    [(\rho, \mathcal{H}, \mathcal{D}_e(1+\mathcal{D}_e^*\mathcal{D}_e)^{-\frac{1}{2}})] = [(\rho, \mathcal{H}, \mathcal{D}(1+\mathcal{D}^*\mathcal{D})^{-\frac{1}{2}})] \quad\text{in}\quad  K^*(J).
\end{equation}
By excision, \(K^*(J\vartriangleleft A)\cong K^*(J)\) for any \(A\), so \eqref{eq:intro_independence_of_extension} give us freedom to choose an extension of \(\mathcal{D}\) when looking for a preimage in \(K^*(A)\) of the \(K\)-homology cycle of a higher-order relative spectral triple. For differential operators, a natural preimage to look for would be realizations corresponding to well-behaved boundary conditions. By exactness of \(K\)-homology, the boundary map \(\partial\) fits into the exact sequence \(\begin{tikzcd}[column sep=small] K^0(C_0(\overline{\Omega})) \ar[r] & K^0(C_0(\Omega)) \ar[r,"\partial"] & K^0(C_0(\partial \Omega))\end{tikzcd}\). Therefore, the image \(\partial [D]\) under the boundary map is an obstruction to the existence of ''well-behaved'' boundary conditions for \(D\), such as Shapiro-Lopatinskii elliptic boundary conditions for classically elliptic differential operators (see discussion towards the end of this section).

For differential operator \(D\) that is elliptic in an appropriate pseudo-differential calculus, under the assumption that \(\ran D^\dagger_{\min}\) is closed, the \(K\)-homology boundary map can be calculated as
\begin{equation}
    \label{eq:intro_boundary_map_image_is_kerDmax}
    \partial [D] = [P_{\Ker D_{\max}}] \in K^1(C_0(\partial \Omega)).
\end{equation}
Here, the class \([P_{\Ker D_{\max}}]\in K^1(C_0(\partial \Omega))\) is defined by the Busby invariant obtained from the orthogonal projection \(P_{\Ker D_{\max}}\) onto the generalized Bergmann space \(\Ker D_{\max}=\setcond{f\in L^2(\Omega; E)}{Df=0 \text{ in }\Omega}\). Noting that \([D]=0\) in \(K^0(C_0(\Omega))\)
if \(D\) is formally self-adjoint, we obtain the following index theorem.

\begin{theorem}
    \label{thm:intro_generalized_bergmann_toeplitz_index_theorem}
    Let \(\Omega\subseteq M\) be an open set with compact boundary in a smooth manifold \(M\) and \(D\colon C^\infty(M; E)\to C^\infty(M; E)\) a formally self-adjoint differential operator that is elliptic in any pseudo-differential calculus presented in \Cref{sec:host_examples}. Assuming that \(\ran D^\dagger_{\min}\) is closed, then
    \begin{equation}
        \ind (P_{\Ker D_{\max}} \alpha P_{\Ker D_{\max}}) = 0
    \end{equation}
    as an operator on \((\Ker D_{\max})^N\) for any \(\alpha\in M_N(C_0(\overline{\Omega}))\) such that \(\alpha|_{\partial \Omega}\) is invertible.

    If \(D_{\min}\) and \(D^\dagger_{\min}\) has compact resolvent, as in the Shubin calculus, \(\partial \Omega\) does not need to be compact and we can let \(\alpha\in M_N(\mathbb{C} + C_0(\overline{\Omega}))\).
\end{theorem}

A similar result to \Cref{thm:intro_generalized_bergmann_toeplitz_index_theorem} would also be obtained for any differential operator that has a ''well-behaved'' boundary conditions that can be used to lift \([D]\in K^0(C_0(\Omega))\) to an element in \(K^0(C_0(\overline{\Omega}))\).

For the special case when \(D\) is elliptic in the classical sense and \(\overline{\Omega}\) is a compact smooth manifold with boundary, we can rewrite \([P_{\Ker D_{\max}}]\) from \eqref{eq:intro_boundary_map_image_is_kerDmax} in terms of boundary data only using the Calderón projector, and we obtain a generalization of the Baum-Douglas-Taylor index theorem in \eqref{eq:intro_BDT_index_theorem} for higher-order differential operators.

\begin{theorem}
    \label{thm:boundary_index_thm}
    Let \(\overline{\Omega}\) be a compact smooth manifold with boundary and \(D\colon C^\infty(\overline{\Omega};E) \newline \to C^\infty(\overline{\Omega};F)\) be a classically elliptic differential operator, then
    \begin{equation}
        \partial [D] = \varphi_*\left([E_+(D)]\cap [S^*\partial \Omega]\right) \quad\text{in}\quad K^1(C(\partial \Omega)),
    \end{equation}
    where \(\varphi\colon S^*\partial \Omega \to \partial \Omega\) is the fiber map, \(E_+(D)\to S^*\partial \Omega\) is a vector bundle relating to the Calderón projector constructed from the principal symbol of \(D\) at \(\partial \Omega\) (see~\Cref{def:calderon_bundle_decomposition}) defining a \(K\)-theory class in \(K_0(C(S^* \partial \Omega))\), \(\cap\) is the cap product (see~\Cref{sec:cap_product}) and \([S^*\partial \Omega]\in K^1(C(S^*\partial \Omega))\) is the fundamental class of \(S^*\partial \Omega\) (see~\Cref{sec:cap_product}) defined from the \(\text{spin}^c\)~-structure induced from \(T^*\partial \Omega\).
\end{theorem}

Note that \Cref{thm:boundary_index_thm} relies on results that are specific for classically elliptic differential operators. From \Cref{thm:boundary_index_thm} and \eqref{eq:intro_boundary_map_image_is_kerDmax} we obtain the following analog to Boutet de Monvel's index theorem.

\begin{theorem}
    \label{thm:intro_maximal_kernel_index_theorem}
    Let \(\overline{\Omega}\) be a compact smooth manifold with boundary and \(D\colon C^\infty(\Omega;E) \newline \to C^\infty(\Omega;F)\) be a classically elliptic differential operator, then
    \begin{equation}
        \ind (P_{\Ker D_{\max}} \alpha P_{\Ker D_{\max}}) = \int_{S^*\partial \Omega} \textup{ch}(\varphi^*(\alpha|_{\partial \Omega})) \wedge \textup{ch}(E_+(D)) \wedge \textup{Td}(S^*\partial \Omega)
    \end{equation}
    as an operator on \((\Ker D_{\max})^N\) for any \(\alpha\in C(\Omega; M_N(\mathbb{C}))\) such that \(\alpha|_{\partial \Omega}\) is invertible, where \(\varphi\colon S^*\partial \Omega \to \partial \Omega\) is the fiber map, \(\textup{ch}(\cdot)\) is the Chern character of a vector bundle and \(\textup{Td}(S^*\partial \Omega)\) is the Todd class of \(S^*\partial \Omega\).
\end{theorem}

Using \Cref{thm:boundary_index_thm}, one can show that under Poincaré duality \(K^1(C(\partial \Omega))\cong K_1(C_0(T^*\partial \Omega))\) the class \(\partial[D]\) is mapped to the Atiyah-Bott obstruction \cite{Atiyah_Bott_1964} of the existence of Shapiro-Lopatinskii elliptic boundary conditions for \(D\). Conversely, if \(B\) is a Shapiro-Lopatinskii elliptic boundary condition for \(D\) then one can show that the bounded transform of the associated realization \(D_B\) defines a \(K\)-homology cycle for \(C(\Omega)\) which implies that \(\partial[D]=0\) by independence of extension of $D$ and exactness of the boundary map. The technical details for proving these statements go beyond the scope of this work and will appear elsewhere.

The content of the article is organized as follows.

In \Cref{sec:geometric_setting} we start by fixing notation for differential operators, defining (non-relative) higher-order spectral triple and providing examples of calculi in which any elliptic operator gives rise to such a triple. We then define a higher-order relative spectral triple and provide an abstract localization procedure of a higher-order spectral triple with operators acting locally with respect to an ideal to produce our main examples of higher-order relative spectral triples.

In \Cref{sec:horst} we present a detailed proof that higher-order relative spectral triples produce relative \(K\)-homology cycles. We also describe a representative of the class of boundary map applied to such relative \(K\)-homology cycles. Throughout \Cref{sec:horst} we keep in mind the prototypical example of a higher-order relative spectral triple constructed in \Cref{sec:geometric_setting}.

In \Cref{sec:even_boundary_map_geometric_setting} we restrict our attention to differential operators elliptic in the classical sense on smooth manifolds with boundary and prove \Cref{thm:boundary_index_thm}. \Cref{sec:even_boundary_map_geometric_setting} can in large be read separately from the previous sections as it only deals with rewriting \([P_{\Ker D_{\max}}]\) from \eqref{eq:intro_boundary_map_image_is_kerDmax} for a classically elliptic differential operator in terms of data on the boundary.

\subsection*{Comments on notation}

Note that the word \emph{elliptic} will be used to mean having a parametrix in a certain pseudo-differential calculus, which is not necessarily the classical pseudo-differential calculus.

When writing \(\norm{ T_1 f}_1 \lesssim \norm{T_2 f }_2\) for \(T_1, T_2\) operators with range in some normed spaces, we mean that there is a \(C>0\) independent of \(f\) such that \(\norm{ T_1 f}_1 \leq C \norm{T_2 f }_2\). For a closed operator \(T\) on a Hilbert space, we let \(\norm{f}_T^2\coloneqq \norm{f}^2 + \norm{Tf}^2\) denote the graph norm of \(T\) with which $\dom T$ is a Hilbert space.

We let \(C_c(X)\) denote compactly supported continuous functions, \(C_0(X)=\overline{C_c(X)}^{\norm{\cdot}_\infty}\) and \(C_b(X)=\setcond{a\in C(X)}{\norm{a}_\infty<\infty}\) where \(\norm{\cdot}_\infty\) is the supremum norm.

We will use \(D\) for differential operators, and \(\mathcal{D}\) for unbounded operators on an abstract Hilbert space.

\subsection*{Acknowledgement}

I wish to thank my Ph.D. supervisor Magnus Goffeng for seeding the ideas for this article, providing enlightening discussions and thoughtful proofreading. I would also like to thank the anonymous referees for their insightful comments that helped improve the paper.

\section{Constructing higher-order relative spectral triples}
\label{sec:geometric_setting}

\subsection{Notation for manifolds and differential operators}

Let \(M\) denote a smooth manifold (without boundary) and fix a smooth volume density on \(M\). If \(M\) is compact we call it a closed manifold. Denote smooth complex valued functions \(M\) as \(C^\infty(M)\) and those with compact support as \(C_c^\infty(M)\).

For Hermitian vector bundles \(E, F\to M\), let \(\Diff(M; E, F)\) denote the set of differential operators with smooth coefficients, or simply \(\Diff(M; E)\) if \(E=F\). Note that a differential operator \(D\in \Diff(M;E,F)\) acts locally in the sense that \(\supp Df\subseteq \supp f\) for any \(C^\infty(M;E)\). As a map 
\begin{equation}
    D\colon C_c^\infty(M;E)\to C_c^\infty(M;F),
\end{equation}
\(D\) defines a continuous linear operator where we endow compactly supported smooth sections of vector bundles with the locally convex topology of test functions. Denote \(D^\dagger\in \Diff(M;F,E)\) for the unique differential operator such that
\begin{equation}
    \label{eq:formal_adjoint_pairing}
    \braket{Df, g}_{L^2(M;F)} = \braket{f, D^\dagger g}_{L^2(M;E)}
\end{equation}
which is obtained through partial integration. Using \(D^\dagger\) we can extend \(D\) to a continuous map
\begin{equation}
    D\colon \mathcal{D}'(M;E)\to \mathcal{D}'(M;F) \quad\text{by}\quad \braket{Df, g}_{L^2(M;F)} \coloneqq  \braket{f, D^\dagger g}_{L^2(M;E)}
\end{equation}
for any \(f\in \mathcal{D}'(M;E)\) and \(g\in C_c^\infty(M;F)\). Noting that the inclusions
\begin{equation}
    C_c^\infty(M;E) \xhookrightarrow{} L^2(M;E) \xhookrightarrow{} \mathcal{D}'(M;E)
\end{equation}
are continuous and dense, we can view \(D\) as an unbounded operator from \(L^2(M;E)\) to \(L^2(M;F)\) with dense domain \(C_c^\infty(M;E)\). Define the minimal extension and maximal extension of \(D\) as 
\begin{equation}
    \dom D_{\min}\coloneqq \overline{C^\infty_c(M;E)}^{\norm{\cdot}_D}
\end{equation}
where \(\norm{\cdot}_D\) denotes the graph norm of \(D\) on \(L^2(M;E)\) and
\begin{equation}
    \dom D_{\max} \coloneqq  \setcond{f\in L^2(M;E)}{Df\in L^2(M;F)},
\end{equation} 
where \(D\) acts in the distributional sense. These are closed extensions of \(D\) as an unbounded operator \(L^2(M;E)\to L^2(M;F)\) and
\begin{equation}
    \label{eq:max_is_adjoint_of_min_of_formal_adjoint}
    D_{\max}=(D^\dagger_{\min})^*
\end{equation}
\cite[Section 1.4]{Grubb_1996}. We say that a realization of \(D\) is a closed extension \(D_e\) of \(D|_{C_c^\infty(M;E)}\) that acts as \(D\) in the distributional sense. Note that \(D_{\min}\subseteq D_e \subseteq D_{\max}\) for any realization \(D_e\) of \(D\).

For an open set \(\Omega\subseteq M\), denote \(\partial \Omega \coloneqq  \overline{\Omega}\setminus \Omega\) for the boundary of \(\Omega\). Note that \(\Omega\) itself is a smooth manifold without boundary, and we can form \(C^\infty(\Omega; E)\) and \(C_c^\infty(\Omega; E)\) where we implicitly use the restriction of vector bundle \(E\) to \(\Omega\). Identifying \(C_c^\infty(\Omega; E)\) as a subset of \(C_c^\infty(M; E)\) by extending by each section zero and using that differential operators act locally, we can restrict elements of \(\Diff(M; E, F)\) to elements of \(\Diff(\Omega; E, F)\). When regarding the maximal and minimal extension of a differential operator together with an open set \(\Omega\), we will mean to the maximal and minimal extension of the element in \(\Diff(\Omega; E, F)\).

We also define sections \(C_c^\infty(\overline{\Omega};E)\) as restrictions of elements in \(C_c^\infty(M;E)\). Note that even though we cannot interpret \(C_c^\infty(\overline{\Omega}; E)\) as a subset of \(C_c^\infty(M; E)\) we can use that differential operators act locally to form the map
\begin{equation}
    D\colon C^\infty(\overline{\Omega};E)\to C^\infty(\overline{\Omega};F)
\end{equation}
from an operator \(D_M\in\Diff(M;E,F)\) as \(Df = D_M f_M|_\Omega\in C^\infty(\Omega;F)\) for \(f\in C^\infty(\Omega;E)\) and \(f_M\in C^\infty(M;E)\) is any function such that \(f_M|_\Omega=f\). 

We define a smooth manifold with boundary as an open set \(\Omega\subseteq M\) with smooth boundary, in the sense that \(\partial \Omega\) is locally the graph of a smooth function. Note that then \(\partial \Omega\) is itself a smooth manifold without boundary and the inclusion \(\partial \Omega\xhookrightarrow{} M\) is an embedding. This definition is equivalent to defining smooth manifolds with boundary using charts mapping to open subsets of the upper half-plane by the process of doubling the manifold, as commented on in \cite[p. 333]{Taylor_2011}.

\subsection{Examples of higher-order spectral triples}
\label{sec:host_examples}

In this section, we present examples of (non-relative) higher-order spectral triples coming from differential operators.

\begin{definition}[Higher-order spectral triple]
    \label{def:host}
    Let \(A\) be a \ensuremath{C^*}-algebra. Define a \emph{higher-order spectral triple} \((\mathcal{A}, \mathcal{H}, \mathcal{D})\) of order \(m > 0\) for \(A\) as consisting a \(\mathbb{Z}/2\)-graded separable Hilbert space \(\mathcal{H}\), an even representation \(\rho\colon A\to \mathbb{B}(\mathcal{H})\), a dense sub-\(*\)-algebras \(\mathcal{A}\subseteq A\) and an odd self-adjoint unbounded operator \(\mathcal{D}\) on \(\mathcal{H}\) such that for each \(a\in \mathcal{A}\)
    \begin{enumerate}
        \item \(\rho(a)\dom \mathcal{D}\subseteq \dom \mathcal{D}\);
        \item \(\rho(a)(1+\mathcal{D}^2)^{-\frac{1}{2}}\in \mathbb{K}(\mathcal{H})\);
        \item  \([\mathcal{D},\rho(a)](1+\mathcal{D}^2)^{-\frac{1}{2}+\frac{1}{2m}}\) extends to an element in \(\mathbb{B}(\mathcal{H})\).
    \end{enumerate}
    An \emph{odd higher-order spectral triple} is defined in the same way except with \(\mathcal{H}\) trivially graded and \(\mathcal{D}\) is not required to be odd.

    By the standard abuse of notation, we will often omit writing out \(\rho\).
\end{definition}

\Cref{def:host} is analogus to a \(\epsilon\)-unbounded \(KK\)-cycles in \cite[Appendix A]{Goffeng_Mesland_2015} for \(\epsilon=\frac{1}{m}\).

The commutator condition in \Cref{def:host} will be understood in examples using complex interpolation of Sobolev spaces. For a continuous inclusion of Banach spaces \(W\subseteq V\) we let \([V,W]_\theta\) denote the complex interpolation spaces of \(V\) and \(W\) for \(0\leq \theta \leq 1\) \cite[Chapter 4, (2.9)]{Taylor_2011}. For any closed densely defined operator \(T\) on a Hilbert space \(\mathcal{H}\)
\begin{equation} 
    \label{eq:range_interpolation_space}
    (1+T^*T)^{-\frac{\theta}{2}}\colon \mathcal{H} \to [\mathcal{H}, \dom T]_\theta
\end{equation}
are isomorphisms of Hilbert spaces for \(0<\theta\leq 1\). The case when \(\theta=1\), then \eqref{eq:range_interpolation_space} is unitary which can be shown directly using norm estimates. For other \(\theta\), this result can be found in \cite[Theorem 3]{Seeley_1971} and \cite[Chapter 4, Proposition 2.2]{Taylor_2011}.

\begin{example}
    \label{ex:classical_diff_op_as_host}
    Let \(M\) be a closed manifold and \(D\in \Diff(M; E)\) be a formally self-adjoint (classically) elliptic differential operator, then we want to show that 
    \begin{equation}
        (C^\infty(M), L^2(M;E), D)
    \end{equation}
    is a higher-order spectral triple. To do this, consider the classical Sobolev spaces \(H^s(M;E)\) for \(s\in \mathbb{R}\) \cite[Chapter I]{Shubin_2001} where \(H^0(M;E) = L^2(M;E)\) which satisfy the complex interpolation property that 
    \([H^s(M;E), H^{s+t}(M;E)]_\theta=H^{s + \theta t}(M;E)\)
    \cite[Proposition 3.1]{Taylor_2011}. If \(D\) is order \(m\) then it extends to a continuous operator 
    \begin{equation}
        D\colon H^s(M;E)\to H^{s-m}(M;E)
    \end{equation}
    for any \(s\in \mathbb{R}\), and similarly each \(a\in C^\infty(M)\) extends to a continuous map on \(H^s(M;E)\) since \(a\) is a differential operator of order zero. Furthermore, the commutator \([D, a]\) is a differential operator of order at most \(m-1\) using the product rule, hence extends to a continuous map 
    \begin{equation}
        \label{eq:commutator_sobolev_spaces}
        [D, a]\colon H^s(M;E)\to H^{s-(m-1)}(M;E).
    \end{equation}
    
    If \(D\) is elliptic (of order \(m\)) then we can construct a parametrix \(Q\) of \(D\), a pseudodifferential operator \(Q\colon C^\infty(M;F)\to C^\infty(M;E)\) that extends to a continuous operator \(Q\colon H^s(M;F)\to H^{s+m}(M;E)\) such that \(1-QD\) is smoothing \cite[Theorem 5.1]{Shubin_2001}. Since \(f=QDf + (1-QD)f\), we can deduce that 
    \begin{equation}
        \label{eq:elliptic_regularity}
        Df\in H^s(M;F) \text{ implies that } f\in H^{s+m}(M;E)
    \end{equation}
    for any \(f\in \mathcal{D}'(M;E)\). The statement \eqref{eq:elliptic_regularity} is called elliptic regularity, and in particular, we obtain that that \(D\) has a unique closed realization since \(H^m(M; E) = \dom D_{\min} = \dom D_{\max}\). Let us identify \(D\) with its unique realization, then we can combine \eqref{eq:range_interpolation_space} and \eqref{eq:commutator_sobolev_spaces} to obtain the commutator condition in \Cref{def:host}. Lastly, by Rellich's theorem \cite[Proposition 3.4]{Taylor_2011} \(H^m(M; E)\subseteq L^2(M; E)\) is a compact inclusion, giving us the compact resolvent condition in \Cref{def:host}.
\end{example}

In other pseudo-differential calculi than the classical one, the same arguments as presented in \Cref{ex:classical_diff_op_as_host} can be used to show that operators elliptic those calculi give rise to higher-order spectral triples. We now present some examples of such pseudo-differential calculi. The details can be found in the references. We reiterate that ellipticity here refers to ellipticity within a certain calculi.

\begin{example}
    Let \(M\) be a Riemannian manifold with bounded geometry, that is, \(M\) has a positive injectivity radius and bounded curvature. Then any formally self-adjoint elliptic properly supported pseudo-differential operator with bounded symbol as defined in \cite{Kordyukov_1991} gives rise to a higher-order spectral triple. Complex interpolation for the Sobolev spaces and a local version of the Rellish theorem is inherited from the classical Sobolev spaces \(H^s(\mathbb{R}^n)\). This calculus includes all differential operators having bounded coefficients with bounded derivatives, such as the Laplace-Beltrami operator for example.
\end{example}

\begin{example}
    Any formally self-adjoint elliptic properly supported operators in the Shubin calculus on \(\mathbb{R}^n\) \cite[Chapter IV]{Shubin_2001} gives rise to a higher-order spectral triple. For differential operators in the Shubin calculus on \(\mathbb{R}^n\), polynomial growth of coefficients counts towards degrees in the calculus. For example, the harmonic oscillator is elliptic in the Shubin calculus. Complex interpolation of the Sobolev spaces can be obtained using the harmonic oscillator.
\end{example}

\begin{example}
    Any formally self-adjoint elliptic properly supported operator in the Heisenberg calculus on a closed filtered manifold (also called a Carnot manifold) \cite{van_Erp_Yuncken_2019} gives rise to a higher-order spectral triple, which is apparent when considering the Sobolev scale introduced in \cite{Dave_Haller_2022}. Complex interpolation of the Sobolev spaces can be obtained using the results in \cite{Dave_Haller_2020}.
\end{example}

\begin{example}
    \label{ex:construct_formally_selfadjoint_diff_op_host}
    For a differential operator \(D\in \Diff(M; E, F)\) that is not necessarily formally self-adjoint, we instead consider the formally self-adjoint differential operator \(\begin{pmatrix}0 & D^\dagger \\ D & 0\end{pmatrix}\in \Diff(M; E\oplus F)\) which is odd with respect to the grading on \(L^2(M; E\oplus F)=L^2(M; E)\oplus L^2(M; F)\) where \(L^2(M; E)\) is positively graded and \(L^2(M; F)\) is negatively graded. If \(D\) is elliptic in one of the pseudo-differential calculi presented previously in this section, then 
    \begin{equation}
        \left(C_c^\infty(M), L^2(M;E\oplus F), \begin{pmatrix}0 & D^\dagger \\ D & 0\end{pmatrix}\right)
    \end{equation}
    is an even higher-order spectral triple for \(C_0(M)\), where we identify \(D\) and \(D^\dagger\) with their unique closed realizations.
\end{example}

\begin{remark}
    \label{remark:extra_commutator_is_compact_condition}
    All of the examples of higher-order spectral triples presented in this section satisfy the additional property that \([\mathcal{D}, a](1+\mathcal{D}^2)^{-\frac{1}{2}}\in \mathbb{K}(\mathcal{H})\) for all \(a\in \mathcal{A}\), which will be needed for the localization procedure presented later. To see this, note that for a properly supported operator \(D\) acting on sections \(C_c^\infty(M; E)\) and a function \(a\in C_c^\infty(M)\), one can find a \(\chi\in C_c^\infty(M)\) such that \([D, a]=\chi[D, a]\), and we obtain compactness from Rellich's theorem (or an appropriate analog in other calculi).
\end{remark}

\subsection{Higher-order relative spectral triples}

In this section, we will define \emph{higher-order relative spectral triples} and present examples. The main type of example is a \emph{localization} of a (non-relative) higher-order spectral triple presented in \Cref{def:localization_horst}, which works well for differential operators as clarified in \Cref{thm:diff_op_horst}. We also present a second type of example in \Cref{ex:dirac_and_higson_corona} which is constructed from a Dirac operator and the Higson compactification.

\begin{definition}[Higher-order relative spectral triple]
    \label{def:higher_order_relative_spectral_triple}
    
    Let \(A\) be a \ensuremath{C^*}-algebra and \(J\vartriangleleft A\) a closed ideal. Define an \emph{even higher-order relative spectral triple} \((\mathcal{J}\vartriangleleft \mathcal{A}, \mathcal{H}, \mathcal{D})\) of order \(m > 0\) for \(J\vartriangleleft A\) as consisting of a \(\mathbb{Z}/2\)-graded separable Hilbert space \(\mathcal{H}\), an even representation \(\rho\colon A\to\mathbb{B}(\mathcal{H})\), two dense sub-\(*\)-algebras \(\mathcal{A}\subseteq A\) and \(\mathcal{J}\subseteq J\cap \mathcal{A}\), and an odd symmetric closed densely defined unbounded operator \(\mathcal{D}\) on \(\mathcal{H}\) satisfying the following conditions:
    \begin{enumerate}
        
        \item \(\rho(a)\dom \mathcal{D}\subseteq \dom \mathcal{D}\) for any \(a\in \mathcal{A}\); \label{item:condition_a_conserves_domain}
        
        \item \(\rho(j)\dom \mathcal{D}^*\subseteq \dom \mathcal{D}\) for any \(j\in \mathcal{J}\); \label{item:condition_j_maps_to_domain}
        
        \item \(\rho(j)(1+\mathcal{D}^*\mathcal{D})^{-\frac{1}{2}}\in \mathbb{K}(\mathcal{H})\) for any \(j\in \mathcal{J}\); \label{item:condition_j_compact_resolvent}
        
        \item \([\mathcal{D},\rho(a)](1+\mathcal{D}^*\mathcal{D})^{-\frac{1}{2}}\in \mathbb{K}(\mathcal{H})\) for any \(a\in\mathcal{A}\); \label{item:condition_D_commutator_a_compact}
    
        \item  \([\mathcal{D},\rho(a)](1+\mathcal{D}^*\mathcal{D})^{-\frac{1}{2}+\frac{1}{2m}}\) extends to an element in \(\mathbb{B}(\mathcal{H})\) for any \(a\in\mathcal{A}\). \label{item:condition_D_commutator_a_bounded}
        
    \end{enumerate}
    
    An \emph{odd higher-order relative spectral triple} is defined in the same way except with \(\mathcal{H}\) trivially graded and \(\mathcal{D}\) is not required to be odd.
    
    If \(m=1\) then we call \((\mathcal{J}\vartriangleleft \mathcal{A}, \mathcal{H}, \mathcal{D})\) a \emph{relative spectral triple}, analogously to relative unbounded Kasparov module introduced in \cite[Definition 2.1]{Forsyth_Goffeng_Mesland_Rennie_2019}. A relative
    unbounded Kasparov module assumes that condition \ref{item:condition_j_compact_resolvent} holds for any \(a\in \mathcal{A}\) and replaces condition \ref{item:condition_D_commutator_a_compact} with that \(\rho(a)(1-P_{\Ker \mathcal{D}^*})(1 - \mathcal{D}\mathcal{D}^*)^{-\frac{1}{2}}\in \mathbb{K}(\mathcal{H})\) for any \(a\in \mathcal{A}\).
    
    By the standard abuse of notation, we will often omit writing out \(\rho\).
\end{definition}

Note that if \(J=A\), \(\mathcal{J}=\mathcal{A}\) and \(\mathcal{D}\) is self-adjoint in \Cref{def:higher_order_relative_spectral_triple}, then without requiring condition \ref{item:condition_D_commutator_a_compact} we recover the definition of a higher-order spectral triple in \Cref{def:host}. Also, note that a higher-order spectral triple of order \(m>0\) is also of order \(m'\) if \(m'>m\).

\begin{remark}
\label{remark:domain_inclusion_is_continuous}
The maps \(a\colon \dom \mathcal{D}\to\dom \mathcal{D}\) and \(j\colon \dom \mathcal{D}^*\to\dom \mathcal{D}\) in conditions \ref{item:condition_a_conserves_domain} and \ref{item:condition_j_maps_to_domain} are automatically continuous with respect to the graph norms. This comes from that if \(A\in \mathbb{B}(\mathcal{H})\) and \(T,S\) are closed densely defined unbounded operators on \(\mathcal{H}\) such that \(A \cdot \dom T\subseteq \dom S\), then \(S A (1+T^*T)^{-\frac{1}{2}}\) is closed and with domain \(\mathcal{H}\), hence bounded by the closed graph theorem \cite[Theorem 2.9]{Brezis_2011} which implies that \(A\colon \dom T \to \dom S\) is continuous.
\end{remark}

\begin{remark}
\label{remark:finite_dim_kernel}
In examples where \((1+\mathcal{D}^*\mathcal{D})^{-\frac{1}{2}}\) is compact, such as those that are constructed from differential operators on pre-compact subsets, \(\mathcal{D}\) is semi-Fredholm, meaning that \(\Ker \mathcal{D}\) is finite-dimensional and \(\ran \mathcal{D}\) is closed. 

Note that if \((1+\mathcal{D}^*\mathcal{D})^{-\frac{1}{2}}\) is compact, then condition \ref{item:condition_D_commutator_a_compact} follows directly from condition \ref{item:condition_D_commutator_a_bounded}.
\end{remark}

The next example will provide a way to localize a higher-order spectral triple to an ideal given that the operator acts locally with respect to this ideal. Such a localization produces a higher-order relative spectral triple and is the main type of example of higher-order relative spectral triples. Concrete examples of localizations for differential operators will be presented in \Cref{thm:diff_op_horst,,ex:construct_formally_selfadjoint_diff_op}.

\begin{definition}
    \label{def:localization_horst}
    For a higher-order spectral triple \((\mathcal{A},\mathcal{H},\mathcal{D})\) we say that \(\mathcal{D}\) is local with respect to a \(*\)-ideal \(\mathcal{J}\) of \(\mathcal{A}\) if for each \(j\in \mathcal{J}\) there is a \(k\in \mathcal{J}\) such that \(j = kj\) and \(\mathcal{D} j = k \mathcal{D} j\). Note that since \(\mathcal{D}\) is self-adjoint we also have that for each \(j\in \mathcal{J}\) there is a \(k\in \mathcal{J}\) such that \(j = jk\) and \(j \mathcal{D} = j \mathcal{D} k\) on \(\dom \mathcal{D}\). 
    
    If \(\mathcal{D}\) is local with respect to \(\mathcal{J}\), then \(\mathcal{D}_\mathcal{J}\coloneq \overline{\mathcal{D}|_{\mathcal{J} \dom \mathcal{D}}}\) can be seen as an operator on \(\mathcal{H}_\mathcal{J}\coloneq \overline{\mathcal{J} \mathcal{H}}\subseteq \mathcal{H}\) and we call the triple
    \begin{equation}
        \label{eq:localization_horst}
        (\mathcal{J}\vartriangleleft \mathcal{A}, \mathcal{H}_\mathcal{J}, \mathcal{D}_\mathcal{J})
    \end{equation}
    the localization of \((\mathcal{A},\mathcal{H},\mathcal{D})\) with respect to \(\mathcal{J}\).
\end{definition}

Another approach to localization in the case when \(m=1\) can be found in \cite[Example 3.23]{Forsyth_Goffeng_Mesland_Rennie_2019} and is discussed further in \Cref{remark:differens_to_FGMR_localization}.

\begin{theorem}
    \label{thm:localization_horst}
    Let \((\mathcal{A},\mathcal{H},\mathcal{D})\) be a higher-order spectral triple for \(A\) such that \(A\) is unital or \([\mathcal{D}, a](1+\mathcal{D}^2)^{-\frac{1}{2}}\in \mathbb{K}(\mathcal{H})\) for all \(a\in \mathcal{A}\). If \(\mathcal{D}\) is local with respect to a \(*\)-ideal \(\mathcal{J}\) of \(\mathcal{A}\), then the localization \((\mathcal{J}\vartriangleleft \mathcal{A}, \mathcal{H}_\mathcal{J}, \mathcal{D}_\mathcal{J})\) of \((\mathcal{A},\mathcal{H},\mathcal{D})\) as in \Cref{def:localization_horst} defines a higher-order relative spectral triple with the additional property that \(a(1+\mathcal{D}_\mathcal{J}^*\mathcal{D}_\mathcal{J})^{-\frac{1}{2}}\in \mathbb{K}(\mathcal{H}_\mathcal{J})\) for all \(a\in \mathcal{A}\).
\end{theorem}

Since \(\mathcal{D}_\mathcal{J}\subseteq \mathcal{D}\), the only condition that is not immediate when proving \Cref{thm:localization_horst} is that \(j\in \mathcal{J}\) maps \(\dom \mathcal{D}_\mathcal{J}^*\) to \(\dom \mathcal{D}_\mathcal{J}\). \Cref{sec:interior_regularity} is dedicated to proving this condition. In case \(m=1\) it is significantly simpler to show this condition as explained in \Cref{remark:firstorder_interior_regularity_simplificatoin}.

Note that if \((\mathcal{J}\vartriangleleft \mathcal{A}, \mathcal{H}_\mathcal{J}, \mathcal{D}_\mathcal{J})\) is a higher-order relative spectral triple for \(J\vartriangleleft A\) and \(\mathcal{H}_\mathcal{J} = \overline{\mathcal{J} \mathcal{H}_J}\), then the action of \(A\) on \(\mathcal{H}_\mathcal{J}\) factors through \(A/K\) for any ideal \(K\vartriangleleft A\) such that \(K\cap J=0\). An appropriate choice of \(K\) could be \(J^\perp = \setcond{k\in A}{kj=0=jk \text{ for all } j\in J}\), the orthogonal complement of \(J\) in \(A\) as a \ensuremath{C^*}-module. In particular, we obtain that 
\begin{equation}
    \label{eq:localization_horst_with_K}
    (\mathcal{J}\vartriangleleft (\mathcal{A}/(J^\perp\cap \mathcal{A})), \mathcal{H}_\mathcal{J}, \mathcal{D}_\mathcal{J})
\end{equation}
is a higher-order relative spectral triple for \(J\vartriangleleft (A/J^\perp)\).  In particlar, if we consider \(C_0(\Omega)\vartriangleleft C_0(M)\) for an open set \(\Omega\subseteq M\) in a smooth manifold \(M\), we have that \(C_0(M)/C_0(\Omega)^\perp =  C_0(\overline{\Omega})\). Also, any \(D\in \Diff(M;E)\) on is local with respect to \(C_c^\infty(\Omega)\) since for any \(j\in C_c^\infty(\Omega)\) we can find a \(k\in C_c^\infty(\Omega)\) such that \(k=1\) on the support of \(j\) and then \(jD=jDk\) on any distribution. As such, we arrive at the motivating example of a higher-order relative spectral triple. Note that the extra property needed in \Cref{thm:localization_horst} is commented on in \Cref{remark:extra_commutator_is_compact_condition}.

\begin{corollary}
    \label{thm:diff_op_horst}
    Let \(M\) be a smooth manifold and \(D\in \Diff(M; E)\) be formally self-adjoint that is elliptic in any pseudo-differential calculus presented in \Cref{sec:host_examples}, then for any open set \(\Omega\subseteq M\) the collection
    \begin{equation}
        (C_c^\infty(\Omega)\vartriangleleft C_c^\infty(\overline{\Omega}), L^2(\Omega;E), D_{\min})
    \end{equation}
    defines a higher-order relative spectral triple of order \(m\) for \(C_0(\Omega)\vartriangleleft C_0(\overline{\Omega})\).
\end{corollary}

Note that in \Cref{thm:diff_op_horst} there is no assumption on the regularity of the boundary \(\partial \Omega\) and that localization procedure does not explicitly depend on the differential operator \(D\) having smooth coefficients.

\begin{example}
    \label{ex:construct_formally_selfadjoint_diff_op}
    Let \(\Omega\subseteq M\) be an open set in a smooth manifold and \(D\in \Diff(M; E, F)\) a differential operator that is elliptic in any pseudo-differential calculus presented in \Cref{sec:host_examples}. We follow the same procedure as in \Cref{ex:construct_formally_selfadjoint_diff_op_host} and then localize as in \Cref{def:localization_horst} to obtain that
    \begin{equation}
        \label{eq:odd_horst_from_diff_op}
        \left(C_c^\infty(\Omega)\vartriangleleft C_c^\infty(\overline{\Omega}), L^2(\Omega;E\oplus F), \begin{pmatrix}0 & D^\dagger_{\min} \\ D_{\min} & 0\end{pmatrix}\right)
    \end{equation}
    is an even higher-order relative spectral triple for \(C_0(\Omega)\vartriangleleft C_0(\overline{\Omega})\). 
    
    We will later denote the relative \(K\)-homology class (see \Cref{thm:horst_defines_kcykle}) obtained from the spectral triple in \eqref{eq:odd_horst_from_diff_op} as    
    \begin{equation}
        [D] \coloneqq  \left[\begin{pmatrix}0 & D^\dagger_{\min} \\ D_{\min} & 0\end{pmatrix}\right] \in K^0(C_0(\Omega)\vartriangleleft C_0(\overline{\Omega})).
    \end{equation}
\end{example}

The next example is fundamentally different from a localization as in \Cref{def:localization_horst}.

\begin{example}
    \label{ex:dirac_and_higson_corona}
    Let \(M\) be a complete Riemannian \(\text{spin}^c\)~ manifold without boundary. Define the \emph{Higson compactification} with respect to the complete metric as compactification of \(M\) corresponding to the unital commutative \ensuremath{C^*}-algebra
    \begin{equation}
        C_h(M)=\setcond{a\in C(M)}{V_r(a)\in C_0(M) \text{ for all } r>0}
    \end{equation}
    where \(V_r\colon C(M)\to C(M)\) is a map defined as 
    \begin{equation}
        V_r(a)(x)=\sup_{y\in B_r(x)} |a(x)-a(y)|
    \end{equation}
    \cite[Chapter 5, Section 5.1]{Roe_1993}.
    Note that \(C_0(M)\subseteq C_h(M)\subseteq C_b(M)\) and \(C_h^\infty(M)\coloneqq \setcond{a\in C_b^\infty(M)}{ da\in C_0^\infty(M;T^*M)}\) is a dense \(*\)-subalgebra of \(C_h(M)\), where \(C_b(\Omega)=\setcond{a\in C(\Omega)}{\norm{a}_\infty<\infty}\).

    Let \(\slashed{D}\in \Diff(M;\slashed{S})\) be the Dirac operator constructed from the complete metric and the \(\text{spin}^c\)~ structure. Then \(\slashed{D}\) has a unique realization that is self-adjoint \cite[Chapter 2, Theorem 5.7]{Lawson_Michelsohn_1989}. Furthermore, \([\slashed{D}, a]\) is Clifford multiplication by \(da\) for any \(a\in C^\infty(M)\) \cite[Chapter 2, Lemma 5.5]{Lawson_Michelsohn_1989}, so if \(da\in C_0^\infty(M;TM)\) we can approximate \([\slashed{D}, a]\) with compactly supported functions and hence \([\slashed{D},a]\colon \dom \slashed{D} \to L^2(M;\slashed{S})\) is compact. We obtain that
    \begin{equation}
        \label{eq:dirac_spectral_triple}
        \left(C_c^\infty(M)\vartriangleleft C_h^\infty(M), L^2(M;\slashed{S}), \overline{\slashed{D}}\right)
    \end{equation}
    defines a relative spectral triple for \(C_0(M)\vartriangleleft C_h(M)\) where \(C_h(M)\) acts on \(L^2(\Omega;E)\) as point-wise multiplication. Note that \(\overline{\slashed{D}}\) is self-adjoint in this example.
    
    For the relative spectral triple \eqref{eq:dirac_spectral_triple}, \(a(1+\mathcal{D}^*\mathcal{D})^{-\frac{1}{2}}\) is not necessarily compact for \(a\in \mathcal{A}\) in contrast to the localization in \Cref{def:localization_horst}. However, if \(\slashed{D}\) is defined from a spin structure and the metric on \(M\) has properly positive scalar curvature, then \(\slashed{D}\) has compact resolvent \cite[Proposition 2.7]{Haskell_Wahl_2009} as a result of Bochner-Lichnerowicz formula \(\slashed{D}^2 = \nabla^\dagger \nabla + \kappa/4\) where \(\kappa\) is the scalar curvature \cite[Theorem 8.8]{Lawson_Michelsohn_1989}.
\end{example}

\begin{example}
    \label{example:odd_odd_tensor_product_of_horst}

    Given two odd higher-order relative spectral triples \((\mathcal{J}_i \vartriangleleft \mathcal{A}_i, \mathcal{H}_i, \mathcal{D}_i)\) for \(i=1,2\), we can construct an even higher-order relative spectral triple as
    \begin{equation}
        \begin{split}
            \label{eq:odd_odd_tensor_product_of_horst}
            &(\mathcal{J}_1 \vartriangleleft \mathcal{A}_1, \mathcal{H}_1, \mathcal{D}_1) \hat{\otimes} (\mathcal{J}_2 \vartriangleleft \mathcal{A}_2, \mathcal{H}_2, \mathcal{D}_2) \\
            &=\left(  \mathcal{J}_1 \otimes_{\textup{alg}} \mathcal{J}_2 \vartriangleleft \mathcal{A}_1\otimes_{\textup{alg}} \mathcal{A}_2, \mathcal{H}_1\otimes \mathcal{H}_2 \otimes \mathbb{C}^2, \mathcal{D} \right)
        \end{split}
    \end{equation}
    where 
    \begin{equation}
        \mathcal{D} = \mathcal{D}_1\otimes \mathds{1}\otimes \sigma_1 + \mathds{1}\otimes \mathcal{D}_2\otimes \sigma_2
    \end{equation}
    and \((\mathcal{H}_1\otimes \mathcal{H}_2)\otimes \mathbb{C}^2\) has grading operator \(\mathds{1} \otimes \mathds{1}\otimes \sigma_3\) for
    \begin{equation}
        \sigma_1 = \begin{pmatrix}0 & -i \\ i & 0\end{pmatrix},~\sigma_2 = \begin{pmatrix}0 & 1 \\ 1 & 0\end{pmatrix} \text{ and } \sigma_3 = \begin{pmatrix}1 & 0 \\ 0 & -1\end{pmatrix}.
    \end{equation}
    The order of \eqref{eq:odd_odd_tensor_product_of_horst} will be the maximum of the orders of the higher-order relative spectral triples it was constructed from. Indeed, \eqref{eq:odd_odd_tensor_product_of_horst} is the exterior Kasparov product and \([\mathcal{D}]=[\mathcal{D}_1]\hat{\otimes}[\mathcal{D}_2]\) in \(K\)-homology.

    Note that the continuity of \(j_i\colon \dom \mathcal{D}_i^* \to \dom \mathcal{D}_i\)  as in \Cref{remark:domain_inclusion_is_continuous} is needed in order to show the domain inclusion \(j \dom \mathcal{D}^* \subseteq \dom \mathcal{D}\) for any \(j\in \mathcal{J}_1 \otimes_{\textup{alg}} \mathcal{J}_2\). This is done by showing that
    \begin{equation}
        \dom \mathcal{D}^* = \overline{(\dom \mathcal{D}_1^*\otimes_{\textup{alg}} \mathcal{H}_2\otimes_{\textup{alg}} \mathbb{C}^2)\cap (\mathcal{H}_1\otimes_{\textup{alg}} \dom \mathcal{D}_2^*\otimes_{\textup{alg}} \mathbb{C}^2)},
    \end{equation}
    with the argument used in \cite[Theorem 1.3]{Deeley_Goffeng_Mesland_2018} and noting that any \(j_1\otimes j_2 \in \mathcal{J}_1 \otimes_{\textup{alg}} \mathcal{J}_2\) preserves the dense subspace \((\dom \mathcal{D}_1^*\otimes_{\textup{alg}} \mathcal{H}_2\otimes_{\textup{alg}} \mathbb{C}^2)\cap (\mathcal{H}_1\otimes_{\textup{alg}} \dom \mathcal{D}_2^*\otimes_{\textup{alg}} \mathbb{C}^2)\).
\end{example}

\subsection{Interior regularity}
\label{sec:interior_regularity}

This section in its entirety will be devoted to proving the following theorem, which is the last ingredient needed to prove that the localization of higher-order spectral triples presented in \Cref{def:localization_horst} indeed produces higher-order relative spectral triples.

\begin{theorem}[Interior regularity]
    \label{thm:interior_regularity}
    Let \((\mathcal{J}\vartriangleleft \mathcal{A}, \mathcal{H}_\mathcal{J}, \mathcal{D}_\mathcal{J})\) be a localization of \((\mathcal{A},\mathcal{H},\mathcal{D})\) as in \Cref{def:localization_horst}, then \(j\cdot \dom \mathcal{D}_\mathcal{J}^* \subseteq \dom \mathcal{D}_\mathcal{J}\) for any \(j\in \mathcal{J}\).
\end{theorem}

\begin{remark}
    The term interior regularity (taken from \cite[p. 434]{Grubb_1968}) usually references to that \(\dom D_{\max}\) for an elliptic differential operator \(D\) on an open set \(\Omega\subseteq M\) has a continuous dense injection into a local Sobolev space, an inverse limit over Sobolev spaces on compact subsets of \(\Omega\). However, this is equivalent to the statement in \Cref{thm:interior_regularity} in the case of differential operators, where the statement says that
    \begin{equation}
        j\colon \dom D_{\max} \to \dom D_{\min}
    \end{equation}
    is well-defined and continuous for any \(j\in C_c^\infty(\Omega)\). Continuity is automatic as in \Cref{remark:domain_inclusion_is_continuous} and from this one obtains a Gårding type inequality. That \(\mathcal{D}_\mathcal{J}^*\) corresponds to \(D_{\max}\) is clarified by \Cref{eq:abstract_maximal_domain_description}.
\end{remark}

To prove \Cref{thm:interior_regularity}, we will construct a sort of parametrix of \(\mathcal{D}\) with enough regularity that is local with respect to a fixed element of \(\mathcal{J}\). To make sense of regularity in this abstract setting we will take inspiration from \cite[Appendix B]{Connes_Moscovici_1995} and introduce a scale of Hilbert spaces \(\mathcal{H}^s\) for \(s\in \mathbb{R}\) induced by \(\mathcal{D}\). That is, let 
\begin{equation}
    \mathcal{H}^s = \dom (1 + \mathcal{D}^2)^\frac{s}{2m}
\end{equation}
for \(s\geq 0\) and let \(\mathcal{H}^{-s}\) be dual of \(\mathcal{H}^s\) identified using the innerproduct of \(\mathcal{H}\), meaning that \(\mathcal{H}^{-s}\) is the completion of \(\mathcal{H}\) in the norm \(\sup_{g\in \mathcal{H}^s\setminus \set{0}} \frac{\abs{\braket{f, g}}}{\norm{g}_{\mathcal{H}^s}}\). Note that by \eqref{eq:range_interpolation_space} together with duality we have that \([\mathcal{H}^s, \mathcal{H}^t]_\theta=\mathcal{H}^{(1-\theta)s+\theta t}\) and that \(\mathcal{H}^{sm} = \dom \mathcal{D}^s\) for \(s=0,1,2,\dots\). Note also that \(\mathcal{H}^\infty\coloneq \bigcap \mathcal{H}^s\) is dense in any \(\mathcal{H}^s\) since \(\dom e^{\mathcal{D}^2}\subseteq \mathcal{H}^\infty\) which is dense using the spectral theorem \cite[Section 7.3]{Weidmann_1980}. 

For \(m\in \mathbb{Z}\), we let \(op^m\) denote the space of all linear operators \(T\) on \(\mathcal{H}^\infty\) that extends to continuous operators \(T\colon\mathcal{H}^s\to \mathcal{H}^{s-m}\) for all \(s\in \mathbb{R}\). Note that \(\bigcup op^m\) is a \(\mathbb{Z}\)-filtered algebra and that \(\mathcal{D}\in op^m\) and \((1+\mathcal{D}^2)^{-1}\in op^{-2m}\). 

Elements of \(\mathcal{A}\) are automatically continuous as operators on \(\dom \mathcal{D}\), but not necessarily on the domain of powers of \(\mathcal{D}\) so \(\mathcal{A}\) cannot necessarily be represented as operators in \(op^0\). To handle the regularity of the action of \(\mathcal{A}\), we introduce operators with a looser connection with the scale \(\mathcal{H}^s\). Namely, let \(op^m_{l,h}\) denote operators \(T\colon \mathcal{H}^h\to \mathcal{H}^{h-m}\) that extend to continuous operators  \(T\colon \mathcal{H}^s\to \mathcal{H}^{s-m}\) for \(l\leq s \leq h\). Note that \(op^{m-1}_{l,h} \subseteq op^m_{l,h+1}\) and hence \(op^m + op^{m-1}_{l,h} \subseteq op^m_{l,h+1}\). As an example of what operators in \(op^m_{l,h}\) could be, \(C^k(M)\subseteq op^0_{-k,k}\) if \(\mathcal{H}^s\) are the classical Sobolev spaces \(H^s(M;E)\) on some closed manifold \(M\).

\begin{lemma}
    Let \((\mathcal{A}, \mathcal{H}, \mathcal{D})\) be a higher-order spectral triple of order \(m\), then \(\mathcal{A}\subseteq op_{-m,m}^0\) and \([\mathcal{D}, \mathcal{A}]\subseteq op_{0,m-1}^{m-1}\).
\end{lemma}

\begin{proof}
    For \(a\in \mathcal{A}\), we have that \(a\) is bounded on \(\dom \mathcal{D} = \mathcal{H}^m\) as in \Cref{remark:domain_inclusion_is_continuous}. Since the same holds for \(a^*\) we can use duality and interpolation to obtain that \(a\in op_{-m,m}^0\). 

    Since \([\mathcal{D}, a](1+\mathcal{D}^2)^{-\frac{1}{2}+\frac{1}{2m}}\) and \(\left([\mathcal{D}, -a^*](1+\mathcal{D}^2)^{-\frac{1}{2}+\frac{1}{2m}}\right)^* = \overline{(1+\mathcal{D}^2)^{-\frac{1}{2}+\frac{1}{2m}}[\mathcal{D}, a]}\) are both bounded, \([\mathcal{D}, a]\) is continuous as a map \(\mathcal{H}^{m-1}\to \mathcal{H}^0\) and \(\mathcal{H}^0\to \mathcal{H}^{-(m-1)}\). Using interpolation this implies that \([\mathcal{D}, a]\in op_{0,m-1}^{m-1}\).
\end{proof}

\begin{lemma}
    \label{thm:localized_parametrix}
    Let \((\mathcal{A}, \mathcal{H}, \mathcal{D})\) be a higher-order spectral triple of order \(m\) and \(\mathcal{J}\subseteq \mathcal{A}\) sub-\(*\)-algebra such that \(\mathcal{D}\) is local with respect to \(\mathcal{J}\) in the sense presented in \Cref{def:localization_horst}. Then for any \(j\in\mathcal{J}\), there is a \(\mathcal{Q}\in op^{-m}_{-m, 0}\) and a \(k\in\mathcal{J}\) such that \(1 - \mathcal{Q} \mathcal{D} \in op^{-m}_{0,0}\) and \(j \mathcal{Q} = j \mathcal{Q} k\).
\end{lemma}

\begin{proof}

    Let \(q=\mathcal{D}(1+\mathcal{D}^2)^{-1}\in op^{-m}\), then \(1-q\mathcal{D}=(1+\mathcal{D}^2)^{-1} \in op^{-m}\). Using that for each \(a\in \mathcal{A}\)
    \begin{equation}
        \begin{split}
            [(1 + \mathcal{D}^2)^{-1},a] 
            =& -\mathcal{D}^2(1+\mathcal{D}^2)^{-1}a(1+\mathcal{D}^2)^{-1} \\ 
            & +(1+\mathcal{D}^2)^{-1}a \mathcal{D}^2(1+\mathcal{D}^2)^{-1} \\
            =& -\mathcal{D}(1+\mathcal{D}^2)^{-1}[\mathcal{D},a](1+\mathcal{D}^2)^{-1} \\ 
            & - (1+\mathcal{D}^2)^{-1} [\mathcal{D}, a]\mathcal{D}(1+\mathcal{D}^2)^{-1}
        \end{split}
    \end{equation}
    we obtain that 
    \begin{equation}
        [q, a] = (1+\mathcal{D}^2)^{-1}[\mathcal{D},a](1+\mathcal{D}^2)^{-1} - q [\mathcal{D}, a] q
    \end{equation}
    from which we can conclude that \([q,a]\in op^{-m-1}_{-m, -1}\). Furthermore, note that \(q + [q, a] \in op^{-m}_{-m, 0}\) and \(1-(q + [q, a]) \mathcal{D}\in op^{-1}_{0,m-1}\subseteq op^0_{0,m}\).

    Fix a \(j_0\in \mathcal{J}\) and inductivly for \(i=1,\dots, 2m-1\) let \(j_i\in \mathcal{J}\) be such that \(j_i=j_i j_{i+1}\) and \(j_i\mathcal{D}=j_i\mathcal{D} j_{i+1}\). Let \(q_i = q + [q, j_{2i-1}]\) for \(i=1, \dots, m\), then \(j_{2i} q_i = j_{2i} q_i j_{2i+1}\) and \(j_{2i} (1-q_i \mathcal{D}) = j_{2i} (1-q_i \mathcal{D}) j_{2i+2}\). We let 
    \begin{equation}
        \mathcal{Q} = \sum_{i=1}^m (1 - q_1 \mathcal{D})\dots (1-q_{i-1}\mathcal{D}) q_i 
    \end{equation}
    then \(\mathcal{Q}\in op^{-m}_{-m,0}\), \(1-\mathcal{Q} \mathcal{D}= (1 - q_1 \mathcal{D})\dots (1-q_m \mathcal{D})\in op^{-m}_{0,0}\) and \(j_0 \mathcal{Q} = j_0 \mathcal{Q} j_{2m-1}\).
\end{proof}

\begin{remark}
    In the case of differential operators on a manifold \(M\) and \(\mathcal{J}=C_c^\infty(\Omega)\), that we can find a parametrix that is local enough as \(\mathcal{Q}\) in \Cref{thm:localized_parametrix} means that given a compact \(K\subseteq \Omega\) we can find a paramterix with Schwartz kernel that is zero on \((M\setminus K')\times K\) for some compact \(K'\subseteq \Omega\).
\end{remark}

Using our notions of regularity with respect to \(\mathcal{D}\) can give a better description of \(\dom \mathcal{D}_\mathcal{J}^*\).

\begin{lemma}
    \label{thm:abstract_maximal_domain_description}
    Let \((\mathcal{J}\vartriangleleft \mathcal{A}, \mathcal{H}_\mathcal{J}, \mathcal{D}_\mathcal{J})\) be a localization of \((\mathcal{A},\mathcal{H},\mathcal{D})\) as in \Cref{def:localization_horst}, then 
    \begin{equation}
        \label{eq:abstract_maximal_domain_description}
        \dom \mathcal{D}_\mathcal{J}^* = \setcond{u\in \mathcal{H}_\mathcal{J}}{j\mathcal{D} u\in \mathcal{H} \text{ and } \norm{j \mathcal{D} u}\lesssim \norm{j}_J \text{ for all } j\in \mathcal{J}}
    \end{equation}
    and \(\mathcal{D}_\mathcal{J}^* u = \lim_n j_n \mathcal{D} u\) for any approximate unit. Here we see \(\mathcal{D}\) as an operator from \(\mathcal{H}\) to \(\mathcal{H}^{-m}=\left(\dom \mathcal{D}\right)^*\) and \(j\in \mathcal{J}\) as operators on \(\mathcal{H}^{-m}\).
\end{lemma}

\begin{proof}

    Unraveling the definitions of the domains, what we want to show is that for \(u\in \mathcal{H}_\mathcal{J}\) the statement
    \begin{equation}
        \abs{\braket{u, \mathcal{D} j v}}\lesssim \norm{j v} \text{ for all } j\in \mathcal{J} \text{ and } v\in \dom \mathcal{D}
    \end{equation}
    is equivalent to the statement
    \begin{equation}
        \abs{\braket{u, \mathcal{D} j^* v}}\lesssim \norm{j}_J \norm{v} \text{ for all } j\in \mathcal{J} \text{ and } v\in \dom \mathcal{D}
    \end{equation}
    which is immediate the assumption made in \Cref{def:localization_horst} that for each \(j\in \mathcal{J}\) there is a \(k\in \mathcal{J}\) such that \(j=kj\).

    Let \(\set{j_n}\subseteq \mathcal{J}\) be a approximate unit and \(u\in \dom \mathcal{D}_\mathcal{J}^*\), then for any \(j\in \mathcal{J}\) and \(v\in \dom \mathcal{D}\)
    \begin{equation}
        \braket{\lim j_n \mathcal{D} u, j v} = \lim \braket{j_n \mathcal{D} u, j v} = \lim  \braket{u, \mathcal{D} j_n^* j v} = \lim  \braket{\mathcal{D}_\mathcal{J}^* u, j_n^* j v} = \braket{\mathcal{D}_\mathcal{J}^* u, j v}.
    \end{equation}
\end{proof}

\begin{proof}[Proof of \Cref{thm:interior_regularity}]
    Fix a \(j\in \mathcal{J}\) and using \Cref{thm:localized_parametrix}, let \(\mathcal{Q}\in op^{-m}_{-m, 0}\) and \(k\in\mathcal{J}\) be such that \(1 - \mathcal{Q}\mathcal{D} \in op^{-m}_{0,0}\) and \(j \mathcal{Q} = j \mathcal{Q} k\). For \(u\in \dom \mathcal{D}_\mathcal{J}^*\), we have that \(u\in \mathcal{H}\) and \(\mathcal{D} u\in \mathcal{H}^{-m}\). Also, by \Cref{eq:abstract_maximal_domain_description} we have that \(k\mathcal{D} u \in \mathcal{H}\) and therefore
    \begin{equation}
        ju=j (1-\mathcal{Q}\mathcal{D}) u + j \mathcal{Q}\mathcal{D} u = j (1-\mathcal{Q}\mathcal{D}) u + j \mathcal{Q} k \mathcal{D} u \in j \dom \mathcal{D} \subseteq \dom \mathcal{D}_\mathcal{J}.
    \end{equation}
\end{proof}

\begin{remark}
    \label{remark:firstorder_interior_regularity_simplificatoin}
    In the case \(m=1\) \Cref{thm:interior_regularity} is considerably easier to show since then \(j\cdot\dom \mathcal{D}_\mathcal{J}^* \subseteq \dom \mathcal{D}\) follows immediately since 
    \begin{equation}
        \abs{\braket{j u, \mathcal{D} v}} = \abs{\braket{u, \mathcal{D} j^* v} - \braket{u, [\mathcal{D}, j^*] v}} \lesssim \norm{\mathcal{D}_\mathcal{J}^* u}~\norm{j^*v} + \norm{u}~\norm{[\mathcal{D}, j^*] v}
    \end{equation}
    for \(u\in \dom \mathcal{D}_\mathcal{J}^*\) and \(v\in \dom \mathcal{D}\). Note that a condition like \(\mathcal{J}^2=\mathcal{J}\) is needed since \(\dom \mathcal{D}_\mathcal{J} \subseteq \dom \mathcal{D} \cap \mathcal{H}_\mathcal{J}\) is not necessarily an equality, which will be explored further later in \Cref{sec:smooth_approximation}.
\end{remark}

\subsection{Smooth approximation}
\label{sec:smooth_approximation}

Let \(\Omega\subseteq M\) be an open set in a smooth manifold \(M\) and let \(D\in \Diff(M;E, F)\). Although it is clear that \(C_c^\infty(\overline{\Omega}; E)\subseteq \dom D_{\max}\), whether the inclusion \(\overline{C_c^\infty(\overline{\Omega}; E)}^{\norm{\cdot}_D} \subseteq \dom D_{\max}\) is an equality or not depends on the regularity of \(\partial \Omega\). We will characterize this property both for an abstract localization of a spectral triple and for differential operators. If \(D\) is elliptic in the classical case, smooth approximation holds under very weak regularity assumptions of \(\partial \Omega\).

\begin{lemma}
    \label{thm:abstract_smooth_approximation}
    Let \((\mathcal{J}\vartriangleleft \mathcal{A}, \mathcal{H}_\mathcal{J}, \mathcal{D}_\mathcal{J})\) be a localization of \((\mathcal{A},\mathcal{H},\mathcal{D})\) as in \Cref{def:localization_horst}, then 
    \begin{equation}
        \dom \mathcal{D}_\mathcal{J}^* = \overline{P_{\mathcal{H}_\mathcal{J}} \dom \mathcal{D}}^{\norm{\cdot}_{\mathcal{D}_\mathcal{J}^*}}
    \end{equation}
    if and only if
    \begin{equation}
        \dom \mathcal{D}_\mathcal{J} = \dom \mathcal{D} \cap \mathcal{H}_\mathcal{J}
    \end{equation}
    where \(P_{\mathcal{H}_\mathcal{J}}\in \mathbb{B}(\mathcal{H})\) is the orthogonal projection on \(\mathcal{H}_\mathcal{J}\).
\end{lemma}

\begin{proof}
    Note that for \(g\in \dom \mathcal{D}\) and \(f\in \dom \mathcal{D}_\mathcal{J}\) 
    \begin{equation}
        \label{eq:projecting_domain}
        \braket{P_{\mathcal{H}_\mathcal{J}} g, \mathcal{D}_\mathcal{J} f} = \braket{g, \mathcal{D}_\mathcal{J} f} = \braket{\mathcal{D} g, f} = \braket{P_{\mathcal{H}_\mathcal{J}} \mathcal{D} g, f}.
    \end{equation}
    which shows that \(P_{\mathcal{H}_\mathcal{J}} \dom \mathcal{D} \subseteq \dom \mathcal{D}_\mathcal{J}^*\) and \(\mathcal{D}_\mathcal{J}^* P_{\mathcal{H}_\mathcal{J}} = P_{\mathcal{H}_\mathcal{J}} \mathcal{D}\) on \(\dom \mathcal{D}\). Let \(\mathcal{D}_e\) be the closed extension of \(\mathcal{D}_\mathcal{J}\) with \(\dom \mathcal{D}_e = \dom \mathcal{D} \cap \mathcal{H}_\mathcal{J}\), then exchanging \(\mathcal{D}_\mathcal{J}\) for \(\mathcal{D}_e\) in \eqref{eq:projecting_domain} shows that \(\dom \mathcal{D}_e^* = \overline{P_{\mathcal{H}_\mathcal{J}} \dom \mathcal{D}}^{\norm{\cdot}_{\mathcal{D}_e^*}}\) which is equivalent to the statement we want to show.
\end{proof}

\begin{remark}
    \label{remark:differens_to_FGMR_localization}
    \Cref{thm:abstract_smooth_approximation} highlights the difference between our approach localization in \Cref{def:localization_horst} and the one found in \cite[Example 3.23]{Forsyth_Goffeng_Mesland_Rennie_2019} for the case when \(m=1\), where they essentially work with \(\mathcal{D}_e\) in the proof on \Cref{thm:abstract_smooth_approximation}.
\end{remark}

The analogous statement for differential operators is the following and is proven similarly as \Cref{thm:abstract_smooth_approximation}.

\begin{lemma}
    \label{thm:smooth_approximation}
    Let \(\Omega\subseteq M\) be an open set and \(D\in \Diff(M;E,F)\) be elliptic of order \(m>0\) in any of the calculi presented in \Cref{sec:host_examples}, then
    \begin{equation}
        \label{eq:smooth_approximation}
        \dom D_{\max} = \overline{C_c^\infty(\overline{\Omega};E)}^{\norm{\cdot}_D}
    \end{equation}
    if and only if
    \begin{equation}
        \label{eq:support_on_domain_condition}
        \dom D^\dagger_{\min} = L^2(\Omega;F)\cap \dom \overline{D^\dagger}
    \end{equation}
    where we see \(D^\dagger\) as an unbounded operator from \(L^2(M;F)\) to \(L^2(M;E)\).
\end{lemma}

The property \eqref{eq:smooth_approximation} in \Cref{thm:smooth_approximation} is called \emph{complete} in \cite[Definition 1.1]{Ballmann_Bar_2012}.

\begin{remark}
    \label{remark:smooth_boundary_has_smooth_approximation}
    The property \eqref{eq:support_on_domain_condition} in \Cref{thm:smooth_approximation} holds in particular for classical Sobolev spaces on compact smooth manifolds with boundary \cite[Chapter 4, Proposition 5.1]{Taylor_2011}, which will be used in \Cref{sec:even_boundary_map_geometric_setting}. It also holds for the much weaker condition that \(\partial \Omega\) is locally the graph of a \(C_0\)-function \cite[Theorem 3.29]{McLean_2000}.
\end{remark}

\section{\texorpdfstring{\(K\)}{K}-homology of higher-order relative spectral triples}
\label{sec:horst}

\subsection{Relative \texorpdfstring{\(K\)}{ K}-homology}

In this subsection, we recall relative \(K\)-homology following \cite{Higson_Roe_2000}.

\begin{definition}[Relative \(K\)-cycles and relative \(K\)-homology]
    Let \(A\) be a \ensuremath{C^*}-algebra and \(J\vartriangleleft A\) be a closed ideal.

    An even relative \(K\)-cycle is a triple \((\rho, \mathcal{H}, F)\) where \(\mathcal{H}\) is a \(\mathbb{Z}/2\)-graded separable Hilbert space, \(\rho\colon A\to\mathbb{B}(\mathcal{H})\) is an even \(*\)-representation and \(F\in \mathbb{B}(\mathcal{H})\) an odd operator satisfying
    \begin{equation}
        \label{eq:relative_Kcycle_conditions}
        \rho(j)(F-F^*),\rho(j)(1-F^2),[F,\rho(a)]\in \mathbb{K}(\mathcal{H})
    \end{equation}
    for any \(j\in J\) and \(a\in A\). We say that a relative \(K\)-cycle is degenerate if all three operators in \eqref{eq:relative_Kcycle_conditions} are zero for any \(j\in J\) and \(a\in A\).

    Let addition of even relative \(K\)-cycles be direct sum. Unitary equivalence of even relative \(K\)-cycles \((\rho,\mathcal{H},F)\sim_U(\rho',\mathcal{H}',F')\) is defined from an even unitary \(U\colon \mathcal{H}\to \mathcal{H}'\) such that \(\rho'(a)=U\rho(a) U^*\) for any \(a\in A\) and \(F'=UFU^*\). To avoid set theoretical issues, we will consider the set of unitary equivalence classes of even relative \(K\)-cycles. An operator homotopy \((\rho,\mathcal{H},F_t)\) is an even relative \(K\)-cycle for each \(t\in [0,1]\) such that \(t\mapsto F_t\) is a norm continuous path.

    Let \(K^0(J\vartriangleleft A)\) be the set of equivalence classes of even relative \(K\)-cycles where \(x\) and \(y\) are equivalent if there are degenerate even relative \(K\)-cycles \(x',y'\) and an operator homotopy \((\rho,\mathcal{H},F_t)\) of even relative \(K\)-cycle such that
    \begin{equation}
        x\oplus x' \sim_U (\rho,\mathcal{H},F_0) \text{ and } (\rho,\mathcal{H},F_1)\sim_U y\oplus y'.
    \end{equation}
    The addition of even relative \(K\)-cycles makes \(K^0(J\vartriangleleft A)\) into a group that we call the even relative \(K\)-homology of \(J\vartriangleleft A\) \cite[Section 8.5]{Higson_Roe_2000}.

    An odd relative \(K\)-cycle is defined in the same way except with \(\mathcal{H}\) trivially graded and \(F\) not required to be odd. The same operations and equivalence relations give rise to the odd relative \(K\)-homology group \(K^1(J\vartriangleleft A)\) of \(J\vartriangleleft A\).

    We will write \(K^*(J\vartriangleleft A)\) when the parity of the relative \(K\)-homology is clear from the context. If \(J=A\), we obtain \(K\)-cycles and \(K\)-homology groups that are not relative. We denote these groups as \(K^*(A)\) which we call \(K\)-homology.
\end{definition}

The advantage of using a relative \(K\)-cycles is that it is feasible to construct the boundary map \(\partial\colon K^*(J\vartriangleleft A)\to K^{*-1}(A/J)\) at the level of cycles. Note that the forgetful map from \(K^*(J\vartriangleleft A)\) to \(K^*(J)\) is an isomorphism by excision \cite[p. 216]{Higson_Roe_2000}. The boundary map will be discussed more later in \Cref{sec:boundary_map}.

A useful way to find another representative of a \(K\)-homology class is by \(J\)-compact perturbations. We say that \(F'\in \mathbb{B}(\mathcal{H})\) is a \(J\)-compact perturbation of \(F\in \mathbb{B}(\mathcal{H})\) if \(\rho(j)(F-F')\in \mathbb{K}(\mathcal{H})\) for all \(j\in J\). If \((\rho, \mathcal{H}, F)\) is a \(K\)-cycle for \(J\) and \(F'\) is a \(J\)-compact perturbation of \(F\), then \((\rho, \mathcal{H}, F')\) is a \(K\)-cycle for \(J\) and \([(\rho, \mathcal{H}, F')]=[(\rho, \mathcal{H}, F)]\) in \(K^*(J)\) using the operator homotopy \((\rho, \mathcal{H}, tF + (1-t)F')\).

\subsection{Odd \texorpdfstring{\(K\)}{ K}-homology and Busby invariants}

In this subsection, we recall Busby invariants and their relation to odd \(K\)-homology. It is based on \cite[Chapter 1 and 2]{Higson_Roe_2000} and \cite[Chapter 3]{Jensen_Thomsen_1991}.

\begin{definition}
    Let \(A\) be a separable \ensuremath{C^*}-algebra.

    A Busby invariant of \(A\) is a \(*\)-homomorphism \(\beta\colon A\to \mathcal{Q}(\mathcal{H})\) for some separable Hilbert space \(\mathcal{H}\), where \(\mathcal{Q}(\mathcal{H})\coloneqq \mathbb{B}(\mathcal{H})/\mathbb{K}(\mathcal{H})\) is the Calkin algebra. Denote the quotient map as \(\pi\colon \mathbb{B}(\mathcal{H})\to\mathcal{Q}(\mathcal{H})\) and say that a Busby invariant \(\beta\) is degenerate if there is a representation \(\rho\) of \(A\) such that \(\beta=\pi\circ \rho\). Note that each Busby invariant corresponds to a short exact sequence \(\begin{tikzcd}[column sep=tiny]0 \ar[r] & \mathbb{K} \ar[r] & E \ar[r] & A \ar[r] & 0\end{tikzcd}\) where degenerate Busby invariants give split exact sequences.
    
    Let addition of Busby invariants be direct sum. Unitary equivalence of Busby invariants \(\beta\sim_U\eta\) is defined from a unitary \(U\in\mathbb{B}(\mathcal{H},\mathcal{H}')\) such that \(\pi(U)\beta(\cdot)=\eta(\cdot)\pi(U)\). To avoid set theoretical issues, we will consider the set of unitary equivalence classes of Busby invariants.

    Let \(\Ext(A)\) be the set of equivalence classes of Busby invariants where \(\beta\) and \(\eta\) are equivalent if there are degenerate Busby invariants \(\beta',\eta'\) such that
    \begin{equation}
        \beta\oplus \beta' \sim_U \eta\oplus \eta'.
    \end{equation}
    The addition of Busby invariants makes \(\Ext(A)\) into a semigroup where zero is the equivalence class of degenerate Busby invariants. Let \(\Ext^{-1}(A)\) denote the group of invertible elements in \(\Ext(A)\), then one can show that \(\Ext^{-1}(A)\cong K^1(A)\) \cite[p. 126]{Higson_Roe_2000}. By the Choi-Effros theorem, \(\Ext(A)=\Ext^{-1}(A)\) if \(A\) is nuclear.
\end{definition}

\begin{lemma}
    \label{thm:similar_busby_invariants}
    Let \(\beta_i\colon A\to\mathcal{Q}(\mathcal{H}_i)\) be Busby invariants for \(i=1,2\). If there is a Fredholm operator \(T\colon \mathcal{H}_1\to\mathcal{H}_2\) such that \(\pi(T)\beta_1(a)=\beta_2(a)\pi(T)\) for any \(a\in A\), where \(\pi\) is the quotient map, then \([\beta_1]=[\beta_2]\) in \(\Ext(A)\).
\end{lemma}

A proof can be found in \cite[Lemma 4.21]{Forsyth_Goffeng_Mesland_Rennie_2019}. The Busby invariants encountered later in the text will come from projections as in the following example.

\begin{example}
    \label{ex:busby_from_projection}
    If \(P\in \mathbb{B}(\mathcal{H})\) is a projection and \(\phi\colon A\to \mathbb{B}(\mathcal{H})\) is a \(*\)-linear map satisfying \([\phi(a),P]\in \mathbb{K}(\mathcal{H})\) and \(P(\phi(ab) - \phi(a)\phi(b))P\in \mathbb{K}(\mathcal{H})\) for all \(a,b\in A\), then \(P\) and \(\phi\) together defines a Busby invariant as
    \begin{equation}
        a \mapsto P \phi(a) P + \mathbb{K}(P \mathcal{H}),
    \end{equation}
    and we denote the class in \(\Ext(A)\) simply as \([P]\) and leave \(\phi\) implicit.
\end{example}

\subsection{Relative \texorpdfstring{\(K\)}{ K}-cycles constructed from higher-order relative spectral triples}

To construct a relative \(K\)-cycle from a higher-order relative spectral triple, we note that for any closed densely defined operator \(\mathcal{D}\) on \(\mathcal{H}\), \((1+\mathcal{D}^*\mathcal{D})^{-\frac{1}{2}}\) is a unitary operator between \(\mathcal{H}\) and \(\dom \mathcal{D}\) with the graph norm. Hence, \(\mathcal{D}(1+\mathcal{D}^*\mathcal{D})^{-\frac{1}{2}}\) is a bounded operator on \(\mathcal{H}\), that we will call the \emph{bounded transform} of \(\mathcal{D}\). The bounded transform retains a lot of information of \(\mathcal{D}\), for example, it has the same kernel and range. 

\begin{theorem}
\label{thm:horst_defines_kcykle}
A higher-order relative spectral triple \((\mathcal{J}\vartriangleleft\mathcal{A},\mathcal{H},\mathcal{D})\) of \(J\vartriangleleft A\) defines a relative \(K\)-cycle \((\rho, \mathcal{H}, \mathcal{D}(1+\mathcal{D}^*\mathcal{D})^{-\frac{1}{2}})\) which we will denote \([\mathcal{D}]\in K^*(J\vartriangleleft A)\) when no confusion arises.
\end{theorem}

The remaining part of this subsection will be devoted to the proof of \Cref{thm:horst_defines_kcykle} which is one of the main results of this paper. However, the proof is not needed for understanding the remaining parts of the paper. The proof is broken into three parts. In \Cref{thm:F_is_almost_selfadjoint} it is proven that \(j(F-F^*)\in \mathbb{K}(\mathcal{H})\) for any \(j\in J\), which then used in \Cref{thm:F_is_almost_idempotent} to prove that \(j(F^2-1)\in \mathbb{K}(\mathcal{H})\). Lastly, it is proven in \Cref{thm:commutator_F_is_compact} that \([F,a]\in \mathbb{K}(\mathcal{H})\) for any \(a\in A\). \Cref{thm:F_is_almost_selfadjoint,,thm:commutator_F_is_compact} are stated more generally for later use.

In the upcoming proofs, the following facts relating to the bounded transform will be used repeatedly without reference.

\begin{lemma}
\label{thm:properties_of_bounded_transform}
Let \(T\) be a closed densely defined operator on a Hilbert space \(\mathcal{H}\) and equip \(\dom T\) with the graph norm of \(T\) making it a Hilbert space. Then for any \(\mu\geq 1\) the following holds:
\begin{enumerate}
    \item \((\mu+T^*T)^{-\frac{1}{2}}\colon \mathcal{H}\to \dom T\) is bounded and \(T(\mu+T^*T)^{-\frac{1}{2}}\in\mathbb{B}(\mathcal{H})\);
    \item \((\mu+T^*T)^{-1} \cdot \dom V = \dom VT^*T\) for any unbounded operator \(V\) on \(\mathcal{H}\);
    \item \(T(\mu+T^*T)^{-\frac{1}{2}}=(\mu+TT^*)^{-\frac{1}{2}}T\) on \(\dom T\) and \(\left(T(\mu+T^*T)^{-\frac{1}{2}}\right)^*=T^*(\mu+TT^*)^{-\frac{1}{2}}\) as elements in \(\mathbb{B}(\mathcal{H})\);
    \item \(\norm{(\mu+T^*T)^{-r}}_{\mathbb{B}(\mathcal{H})}\leq \mu^{-r}\) for any \(r>0\).
\end{enumerate}
\end{lemma}

\begin{proof}
    The first two statements can be shown directly, and the last two can be shown using functional calculus.
\end{proof}

The main tool in the proofs of \Cref{thm:F_is_almost_selfadjoint,,thm:commutator_F_is_compact} will be the integral formula for fractional powers 
\begin{equation}
    (1+T^*T)^{-\frac{1}{2}} = \frac{1}{\pi} \int_1^\infty (\mu-1)^{-\frac{1}{2}}(\mu+T^*T)^{-1}~d\mu
\end{equation}
originating from \cite{Baaj_Julg_1983}, which is shown by functional calculus and residue calculus. How it is used is somewhat delicate. For \(x\in \dom T\), 
\begin{equation}
    \left(\frac{1}{\pi} \int_1^\infty (\mu-1)^{-\frac{1}{2}}(\mu+TT^*)^{-1}~d\mu\right) Tx =
    \frac{1}{\pi} \int_1^\infty (\mu-1)^{-\frac{1}{2}}(\mu+TT^*)^{-1}Tx~d\mu
\end{equation}
and therefore
\begin{equation}
    \label{eq:integral_formula}
    T(1+T^*T)^{-\frac{1}{2}} = \frac{1}{\pi} \int_1^\infty (\mu-1)^{-\frac{1}{2}}T(\mu+T^*T)^{-1}~d\mu
\end{equation}
holds point-wise on \(\dom T\). The integral in \eqref{eq:integral_formula} does not necessarily converge in operator norm and only makes sense point-wise, but using such point-wise expressions on a dense subspace and arriving at a norm convergent integral would show equality as bounded operators.

\begin{lemma}
\label{thm:F_is_almost_selfadjoint}
Let \((\mathcal{J}\vartriangleleft\mathcal{A},\mathcal{H},\mathcal{D})\) be a higher-order relative spectral triple for \(J\vartriangleleft A\) of order \(m\) and \(\mathcal{D}_e\) any extension of \(\mathcal{D}\) satisfying \(\mathcal{D}\subseteq \mathcal{D}_e\subseteq \mathcal{D}^*\), then
\begin{equation}
    j(F - F_e)\in \mathbb{K}(\mathcal{H}) \quad\text{for all } j\in J
\end{equation}
where \(F = \mathcal{D}(1+\mathcal{D}^*\mathcal{D})^{-\frac{1}{2}}\) and \(F_e = \mathcal{D}_e(1+\mathcal{D}_e^*\mathcal{D}_e)^{-\frac{1}{2}}\).

In particular, if \(\mathcal{D}_e=\mathcal{D}^*\) we obtain that \(j(F - F^*)\in \mathbb{K}(\mathcal{H})\) for all \(j\in J\).
\end{lemma}

The following proof is based on \cite[Lemma 3.1]{Hilsum_2010}, where the case \(m=1\) was considered.

\begin{proof}
By density, it suffices to consider \(j\in \mathcal{J}\). We will show that \((F_e^*-F^*)j\in \mathbb{K}(\mathcal{H})\), which is equivalent. Using the integral formula \eqref{eq:integral_formula},
\begin{equation}
    (F_e^*-F^*)j = \frac{1}{\pi} \int_1^\infty (\mu-1)^{-\frac{1}{2}} (\mathcal{D}_e^*(\mu + \mathcal{D}_e\mathcal{D}_e^*)^{-1} - \mathcal{D}^*(\mu+\mathcal{D}\mathcal{D}^*)^{-1})j ~d\mu
\end{equation}
point-wise on \(\dom \mathcal{D}\). If we show that the integral is norm convergent and the integrand is compact, we are done.

Let 
\begin{equation}
    R_\mu \coloneqq  (\mathcal{D}_e^*(\mu + \mathcal{D}_e\mathcal{D}_e^*)^{-1} - \mathcal{D}^*(\mu+\mathcal{D}\mathcal{D}^*)^{-1})j,
\end{equation}
then
\begin{equation}
\begin{split}
\label{eq:F_almost_selfadjoint_first_rewrite}
    R_\mu
    =& \mathcal{D}_e^*(\mu + \mathcal{D}_e\mathcal{D}_e^*)^{-1}(\mu+\mathcal{D}\mathcal{D}^*)(\mu+\mathcal{D}\mathcal{D}^*)^{-1}j \\
    &- (\mu + \mathcal{D}_e^*\mathcal{D}_e)(\mu + \mathcal{D}_e^*\mathcal{D}_e)^{-1}\mathcal{D}^*(\mu+\mathcal{D}\mathcal{D}^*)^{-1}j \\
    =& \mu \mathcal{D}_e^*(\mu + \mathcal{D}_e\mathcal{D}_e^*)^{-1}(\mu+\mathcal{D}\mathcal{D}^*)^{-1}j \\
    &- \mu (\mu + \mathcal{D}_e^*\mathcal{D}_e)^{-1}\mathcal{D}^*(\mu+\mathcal{D}\mathcal{D}^*)^{-1}j
\end{split}
\end{equation}
using that \((\mu + \mathcal{D}_e\mathcal{D}_e^*)^{-1}\mathcal{D}=\mathcal{D}_e(\mu + \mathcal{D}_e^*\mathcal{D}_e)^{-1}\) on \(\dom \mathcal{D}=\ran \mathcal{D}^*(\mu+\mathcal{D}\mathcal{D}^*)^{-1}\). Similarly, \((\mu + \mathcal{D}_e^*\mathcal{D}_e)^{-1}\mathcal{D}^*=\mathcal{D}_e^*(\mu + \mathcal{D}_e\mathcal{D}_e^*)^{-1}\) on \(j \dom \mathcal{D}^* \subseteq \dom \mathcal{D}\). Hence,
\begin{equation}
\begin{split}
    0 =& \mu \mathcal{D}_e^*(\mu + \mathcal{D}_e\mathcal{D}_e^*)^{-1}j(1+\mathcal{D}\mathcal{D}^*)^{-1} \\
    &- \mu(\mu + \mathcal{D}_e^*\mathcal{D}_e)^{-1}\mathcal{D}^*j(1+\mathcal{D}\mathcal{D}^*)^{-1}
\end{split}
\end{equation} 
which we subtract from \(R_\mu\) to obtain
\begin{equation}
\begin{split}
    R_\mu 
    =& \mu \mathcal{D}_e^*(\mu+\mathcal{D}_e\mathcal{D}_e^*)^{-1}[(\mu+\mathcal{D}\mathcal{D}^*)^{-1},j] \\
    &- \mu (\mu+\mathcal{D}_e^*\mathcal{D}_e)^{-1}\mathcal{D}^*[(\mu+\mathcal{D}\mathcal{D}^*)^{-1},j].
\end{split}
\end{equation}
If we rewrite
\begin{equation} 
\label{eq:commutator_inverse_and_operator} 
\begin{split}
    [(\mu+\mathcal{D}\mathcal{D}^*)^{-1},j] 
    =& -\mathcal{D}\mathcal{D}^*(\mu+\mathcal{D}\mathcal{D}^*)^{-1}j(\mu+\mathcal{D}\mathcal{D}^*)^{-1} \\ 
    &+(\mu+\mathcal{D}\mathcal{D}^*)^{-1}j \mathcal{D}\mathcal{D}^*(\mu+\mathcal{D}\mathcal{D}^*)^{-1} \\
    =& -\mathcal{D}(\mu+\mathcal{D}^*\mathcal{D})^{-1}[\mathcal{D}^*,j](\mu+\mathcal{D}\mathcal{D}^*)^{-1} \\
    & - (\mu+\mathcal{D}\mathcal{D}^*)^{-1}[\mathcal{D},j]\mathcal{D}^*(\mu+\mathcal{D}\mathcal{D}^*)^{-1}
\end{split}
\end{equation}
and gather the terms of \(R_\mu\) as
\begin{equation}
    R_\mu = T_1 [\mathcal{D}^*,j](\mu+\mathcal{D}\mathcal{D}^*)^{-1} + T_2 [\mathcal{D},j]\mathcal{D}^*(\mu+\mathcal{D}\mathcal{D}^*)^{-1},
\end{equation}
we see that
\begin{equation}
\begin{split}
    T_1 
    =& -\mu \mathcal{D}_e^*(\mu+\mathcal{D}_e\mathcal{D}_e^*)^{-1}\mathcal{D}(\mu+\mathcal{D}^*\mathcal{D})^{-1} \\
    &+ \mu(\mu+\mathcal{D}_e^*\mathcal{D}_e)^{-1}\mathcal{D}^*\mathcal{D}(\mu+\mathcal{D}^*\mathcal{D})^{-1} = 0
\end{split}
\end{equation}
by again using that \((\mu + \mathcal{D}_e\mathcal{D}_e^*)^{-1}\mathcal{D}=\mathcal{D}_e(\mu + \mathcal{D}_e^*\mathcal{D}_e)^{-1}\) on \(\dom \mathcal{D}\), and  
\begin{equation}
\begin{split}
    T_2 
    =& -\mu \mathcal{D}_e^*(\mu+\mathcal{D}_e\mathcal{D}_e^*)^{-1}(\mu+\mathcal{D}\mathcal{D}^*)^{-1} \\
    &+ \mu(\mu+\mathcal{D}_e^*\mathcal{D}_e)^{-1}\mathcal{D}^*(\mu+\mathcal{D}\mathcal{D}^*)^{-1} \\
    =& -(\mathcal{D}_e^*(\mu + \mathcal{D}_e\mathcal{D}_e^*)^{-1} - \mathcal{D}^*(\mu+\mathcal{D}\mathcal{D}^*)^{-1})
\end{split}
\end{equation}
by noting the similarity to \eqref{eq:F_almost_selfadjoint_first_rewrite} and reversing that computation. We arrive at
\begin{equation}
    R_\mu = -(\mathcal{D}_e^*(\mu + \mathcal{D}_e\mathcal{D}_e^*)^{-1} - \mathcal{D}^*(\mu+\mathcal{D}\mathcal{D}^*)^{-1}) [\mathcal{D},j]\mathcal{D}^*(\mu+\mathcal{D}\mathcal{D}^*)^{-1}
\end{equation}
which is compact by condition \ref{item:condition_D_commutator_a_compact} in \Cref{def:higher_order_relative_spectral_triple} since
\begin{equation}
    [\mathcal{D},j]\mathcal{D}^*(\mu+\mathcal{D}\mathcal{D}^*)^{-1} = [\mathcal{D},j](\mu+\mathcal{D}^*\mathcal{D})^{-\frac{1}{2}} \mathcal{D}^*(\mu+\mathcal{D}\mathcal{D}^*)^{-\frac{1}{2}}.
\end{equation}
Furthermore, from
\begin{equation}
\label{eq:commutator_j_and_Dadj_stuff_expansion}
    [\mathcal{D},j](\mu+\mathcal{D}^*\mathcal{D})^{-\frac{1}{2}}=[\mathcal{D},j](\mu+\mathcal{D}^*\mathcal{D})^{-\frac{1}{2}+\frac{1}{2m}} (\mu+\mathcal{D}^*\mathcal{D})^{-\frac{1}{2m}}
\end{equation}
we obtain the norm estimate \(\norm{ R_\mu } \lesssim \mu^{-\frac{1}{2}-\frac{1}{2m}}\), and we are done.
\end{proof}

\begin{lemma}
\label{thm:F_is_almost_idempotent}
Let \((\mathcal{J}\vartriangleleft\mathcal{A},\mathcal{H},\mathcal{D})\) be a higher-order relative spectral triple for \(J\vartriangleleft A\), then \(j(1-F^2)\in \mathbb{K}(\mathcal{H})\) for all \(j\in J\), where \(F=\mathcal{D}(1+\mathcal{D}^*\mathcal{D})^{-\frac{1}{2}}\).
\end{lemma}

\begin{proof}
By density, it suffices to consider \(j\in \mathcal{J}\). We have that
\begin{equation}
    j(1-F^2) = j(1-F^*F) - j(F-F^*)F\in \mathbb{K}(\mathcal{H})
\end{equation}
using that \(1-F^*F = (1+\mathcal{D}^*\mathcal{D})^{-1}\) and \Cref{thm:F_is_almost_selfadjoint}.
\end{proof}

\begin{lemma}
\label{thm:commutator_F_is_compact}
Let \((\mathcal{J}\vartriangleleft\mathcal{A},\mathcal{H},\mathcal{D})\) be a higher-order relative spectral triple for \(J\vartriangleleft A\), then \([\mathcal{D}(\delta+\mathcal{D}^*\mathcal{D})^{-\frac{1}{2}},a]\in \mathbb{K}\) for all \(a\in A\) and any \(\delta>0\).

In particular, \([F,a]\in \mathbb{K}\) for all \(a\in A\) where \(F=\mathcal{D}(1+\mathcal{D}^*\mathcal{D})^{-\frac{1}{2}}\).
\end{lemma}

A proof when \(m=1\) can be found in \cite[Lemma B.1]{Forsyth_Goffeng_Mesland_Rennie_2019}.

\begin{proof} 

By density, it suffices to consider \(a\in \mathcal{A}\). With a change of variables in the functional calculus calculation used in the integral formula \eqref{eq:integral_formula},
\begin{equation}
    [\mathcal{D}(\delta+\mathcal{D}^*\mathcal{D})^{-\frac{1}{2}},a] = \quad \frac{1}{\pi}\int_{\delta}^\infty (\mu-\delta)^{-\frac{1}{2}}[\mathcal{D}(\mu+\mathcal{D}^*\mathcal{D})^{-1},a]  d\mu
\end{equation}
point-wise on \(\dom \mathcal{D}\). If we show that the integral is norm convergent and the integrand is compact, we are done.

Taking into account that \(\mathcal{D}^*a\mathcal{D}\) is not defined on \(\dom \mathcal{D}^*\mathcal{D}\) in general, we can similarly as in \eqref{eq:commutator_inverse_and_operator} rewrite
\begin{equation}
\begin{split}
    \label{eq:commutator_inverse_and_operator_a}
    [(\mu +\mathcal{D}^*\mathcal{D})^{-1},a] 
    =& -\mathcal{D}^*\mathcal{D}(\mu+\mathcal{D}^*\mathcal{D})^{-1}a(\mu+\mathcal{D}^*\mathcal{D})^{-1} \\ 
    & +(\mu+\mathcal{D}^*\mathcal{D})^{-1}a \mathcal{D}^*\mathcal{D}(\mu+\mathcal{D}^*\mathcal{D})^{-1} \\
    =& -\mathcal{D}^*(\mu+\mathcal{D}\mathcal{D}^*)^{-1}[\mathcal{D},a](\mu+\mathcal{D}^*\mathcal{D})^{-1} \\ 
    & + \left( \mathcal{D}^*(\mu+\mathcal{D}\mathcal{D}^*)^{-1} [\mathcal{D}, a^*](\mu+\mathcal{D}^*\mathcal{D})^{-1} \right)^*
\end{split}
\end{equation}
and hence
\begin{equation}
\begin{split}
    \label{eq:bounded_transform_commutator_expression}
    [\mathcal{D}(\mu+\mathcal{D}^*\mathcal{D})^{-1},a] 
    =& [\mathcal{D},a](\mu+\mathcal{D}^*\mathcal{D})^{-1} + \mathcal{D}[(\mu+\mathcal{D}^*\mathcal{D})^{-1},a] \\ 
    =& \mu(\mu+\mathcal{D}\mathcal{D}^*)^{-1}[\mathcal{D},a](\mu+\mathcal{D}^*\mathcal{D})^{-1} \\ 
    & + \left( \mathcal{D}^*(\mu+\mathcal{D}\mathcal{D}^*)^{-1}[\mathcal{D}, a^*](\mu+\mathcal{D}^*\mathcal{D})^{-1} \mathcal{D}^*\right)^*.
\end{split}
\end{equation}
Now we see that \eqref{eq:bounded_transform_commutator_expression} is compact and has a norm estimate that makes the integral norm convergent in the same way as at the end of the proof of \Cref{thm:F_is_almost_selfadjoint}.
\end{proof}

\begin{remark}
\label{remark:a_domDadj_notin_domD}

In the first order case, that is \(m=1\), we have that \(a\dom \mathcal{D}^* \subseteq \dom \mathcal{D}^*\) for any \(a\in \mathcal{A}\), which can be shown from that \([\mathcal{D},a]\) is bounded. For higher orders, this inclusion does not necessarily hold and makes \eqref{eq:commutator_inverse_and_operator_a} slightly more subtle than \eqref{eq:commutator_inverse_and_operator}. 

For example, the property \(a\dom \mathcal{D}^*\subseteq \dom \mathcal{D}^*\) fails for \(\mathcal{D}=\Delta_{\min}\) on \(L^2(\Omega)\) and \(\mathcal{A}=C^\infty(\Omega)\) where \(\Omega\subseteq \mathbb{C}\) is the unit disc. Explicitly, we can see that it fails for \(-\log(1-z)+\log(1-\bar{z})\in \Ker \Delta_{\max}\) and \(|z|^2\in C^\infty(\Omega)\) since
\begin{equation}
    \Delta |z|^2(-\log(1-z)+\log(1-\bar{z})) = -4\left(\frac{z - \bar{z}}{|1-z|^2} - \log(1-z) + \log(1-\bar{z})\right)
\end{equation}
is not in \(L^2(\Omega)\).
\end{remark}

\begin{remark}
    The definition of a higher-order spectral triple is a weakening of the requirement of bounded commutator in a spectral triple. We could weaken condition \ref{item:condition_D_commutator_a_bounded} further and still prove \Cref{thm:horst_defines_kcykle} with the same methods by assuming that for each \(a\in \mathcal{A}\) there is a Borel measurable function \(f\colon [0,\infty)\to\mathbb{R}\) satisfying \(\int_1^\infty (\mu-1)^{-\frac{1}{2}}\mu^{-\frac{1}{2}} \sup_{x\geq 0} \left|\frac{1}{f(x)}(\mu+x)^{-\frac{1}{2}}\right| d\mu < \infty\) such that \([\mathcal{D},a]f(\mathcal{D}^*\mathcal{D})\) extends to an element in \(\mathbb{B}(\mathcal{H})\). For example, \(f(x) = (1+x)^{-\frac{1}{2}} \log(2+x)^2\) would work.
\end{remark}

\subsection{The \texorpdfstring{\(K\)}{ K}-homology boundary map}
\label{sec:boundary_map}

Given a short exact sequence of separable \ensuremath{C^*}-algebras
\begin{equation}
    \begin{tikzcd}
    0 \ar[r] & J \ar[r] & A \ar[r] & A/J \ar[r] & 0
    \end{tikzcd}
\end{equation}
such that \(A\to A/J\) has a completely positive right-inverse \cite[Definition 3.1.1]{Higson_Roe_2000}, the boundary map in \(K\)-homology is a map \(\partial\) fitting in the exact sequence
\begin{equation}
    \begin{tikzcd}
    \cdots \ar[r] & K^*(A) \ar[r] & K^*(J) \ar[r,"\partial"] & K^{*-1}(A/J) \ar[r] & K^{*-1}(A) \ar[r] & \cdots 
    \end{tikzcd}
\end{equation}
See \cite{Higson_Roe_2000} for further description of the boundary map.

Although the boundary map is defined on \(x\in K^*(J)\), it is generally not feasible to calculate \(\partial x\in K^{*-1}(A/J)\) without finding a representative in the equivalence class \(x\) which is a relative \(K\)-cycle, and hence defines a class in \(K^*(J\vartriangleleft A)\). This is the core reason why we consider a relative version of higher-order spectral triples.

This subsection will describe the boundary map applied to even higher-order relative spectral triples in an abstract setting. Later in \Cref{sec:even_boundary_map_geometric_setting} we will further dismantle this image in the geometric setting for a classically elliptic differential operator on a compact smooth manifold with boundary.

For an odd operator \(T\) on a \(\mathbb{Z}/2\)-graded Hilbert \(\mathcal{H}\), we will write \(T_\pm\) for the operators in \(T=\begin{pmatrix}0 & T_- \\ T_+ & 0\end{pmatrix}\) with respect to the decomposition \(\mathcal{H}=\mathcal{H}_+\oplus \mathcal{H}_-\). That is, the sign denotes what part of the Hilbert space that the operator maps from. Note that \((T_-)^* = (T^*)_+\).

\begin{lemma}
    \label{thm:description_of_boundary_cycle_for_selfadjointpartialisometry}
    Let \(J\vartriangleleft A\) be separable \ensuremath{C^*}-algebras such that \(A\to A/J\) has a completely positive right-inverse and \((\rho,\mathcal{H},F)\) be an even relative \(K\)-cycle of \(J\vartriangleleft A\) such that \(F\) is a self-adjoint partial isometry. Then
    \begin{equation}
        \partial [(\rho,\mathcal{H},F)] = [P_{\Ker F_+}] - [P_{\Ker F_-}] \quad\text{in}\quad K^1(A/J)
    \end{equation}
    where the classes in \(\Ext^{-1}(A/J)\cong K^1(A/J)\) are constructed as in \Cref{ex:busby_from_projection} using \(\rho_\pm \circ s\) respectively for any \(*\)-linear right-inverse \(s\colon A/J \to A\) of the quotient map.
\end{lemma}

This can be found in \cite[Remark 8.5.7]{Higson_Roe_2000} as well as in \cite[p. 784]{Baum_Douglas_Taylor_1989}.

\begin{lemma}
    \label{thm:phase_is_limit}
    Let \(T\) be a closed densely defined operator from \(\mathcal{H}_1\) to \(\mathcal{H}_2\) with closed range, then
    \begin{equation}
        V = \lim_{\delta\to 0} T(\delta+T^*T)^{-\frac{1}{2}}
    \end{equation}
    converges in operator norm, where \(V\) is the partial isometry in the polar decomposition \(T=V|T|\) with \(\Ker V = \Ker T\).
\end{lemma}

\begin{proof}

    Note that \(\Ker |T| = \Ker T\), \(\ran |T| = \ran T^*\) and that if \(\ran T\) is closed, then so is \(\ran T^*=\ran |T|\) \cite[Theorem 2.19]{Brezis_2011}. By \cite[Theorem 2.20]{Brezis_2011} \(|T|\) has a lower bound \(C>0\) on \(\ran T^*=(\Ker T)^\perp\). Restricting to \((\Ker T)^\perp\), by functional calculus
    \begin{equation}
        \norm{1 - |T|(\delta+T^*T)^{-\frac{1}{2}}} \leq \sup_{x\geq C} |1-x(\delta + x^2)^{-\frac{1}{2}}| \leq \frac{\delta}{2 C^2}
    \end{equation}
    and hence
    \begin{equation}
        (1-P_{\Ker T}) = \lim_{\delta\to 0} |T|(\delta+T^*T)^{-\frac{1}{2}} (1-P_{\Ker T})
    \end{equation}
    converges in operator norm on \(\mathcal{H}_1\). Now, 
    \begin{equation}
        V = V (1-P_{\Ker T}) = V \lim_{\delta\to 0} |T|(\delta+T^*T)^{-\frac{1}{2}} (1-P_{\Ker T}) = \lim_{\delta\to 0} T(\delta+T^*T)^{-\frac{1}{2}}
    \end{equation}
    where the last equality uses that \(T=V|T|\) and \(P_{\Ker T}=\delta^{\frac{1}{2}}(\delta + T^*T)^{-\frac{1}{2}}P_{\Ker T}\).
\end{proof}

\begin{lemma}
\label{thm:boundary_map_of_even_horst}
Let \(J\vartriangleleft A\) be separable \ensuremath{C^*}-algebras such that \(A\to A/J\) has a completely positive right-inverse and \((\mathcal{J}\vartriangleleft\mathcal{A},\mathcal{H},\mathcal{D})\) be an even higher-order relative spectral triple for \(J\vartriangleleft A\). If \(\ran \mathcal{D}_-\) is closed, then
\begin{equation}
    \partial [\mathcal{D}] = [P_{\Ker ((\mathcal{D}_-)^*)}] - [P_{\Ker (\mathcal{D}_-)}] \quad\text{in}\quad K^1(A/J),
\end{equation}
or if \(\ran \mathcal{D}_+\) is closed, then
\begin{equation}
    \partial [\mathcal{D}] = [P_{\Ker (\mathcal{D}_+)}] - [P_{\Ker ((\mathcal{D}_+)^*)}] \quad\text{in}\quad K^1(A/J),
\end{equation}
where the classes in \(\Ext^{-1}(A/J)\cong K^1(A/J)\) are constructed as in \Cref{ex:busby_from_projection} using \(\rho_\pm \circ s\) respectively for any \(*\)-linear right-inverse \(s\colon A/J \to A\) of the quotient map. 
\end{lemma}

\begin{proof}
Let \(V\in \mathbb{B}(\mathcal{H})\) be the partial isometry in the polar decomposition \(\mathcal{D} = V|\mathcal{D}|\) with \(\Ker V = \Ker \mathcal{D}\). For the bounded transform \(F=\mathcal{D}(1+\mathcal{D}^*\mathcal{D})^{-\frac{1}{2}}\), noting that \(F=V|F|\) and \(V-F=V(1+|F|)^{-1}(1-F^*F)\) we see that \(V\) is a \(J\)-compact perturbation of \(F\). Also, \(W=\begin{pmatrix}0 & V_- \\ (V_-)^* & 0\end{pmatrix}\) is a self-adjoint partial isometry that is a \(J\)-compact perturbation of \(V\). Hence, we obtain that \([\mathcal{D}]=[(\rho, \mathcal{H}, W)]\) in \(K^0(J)\).

In the case when \(\ran \mathcal{D}_-\) is closed, we can use \Cref{thm:phase_is_limit} to obtain that
\begin{equation}
    [V_-,a] = \lim_{\delta \to 0} [\mathcal{D}_-(\delta + \mathcal{D}_-^*\mathcal{D}_-)^{-\frac{1}{2}}, a]
\end{equation}
converges in operator norm for any \(a\in A\), and is hence compact by \Cref{thm:commutator_F_is_compact}. Hence, \((\rho, \mathcal{H}, W)\) is a relative \(K\)-cycle for \(J\vartriangleleft A\) and 
\begin{equation}
    \partial [\mathcal{D}] = \partial[(\rho, \mathcal{H}, W)] = [P_{\Ker W_+}] - [P_{\Ker W_-}] = [P_{\Ker ((\mathcal{D}_-)^*)}] - [P_{\Ker (\mathcal{D}_-)}]
\end{equation}
using \Cref{thm:description_of_boundary_cycle_for_selfadjointpartialisometry}. The case where \(\ran D_+\) is closed is proven analogously.
\end{proof}

\begin{remark}
\label{remark:a_compact_resolvent_kills_one_term_of_boundary}
With the additional property of a higher-order relative spectral triple that \(a(1+\mathcal{D}^*\mathcal{D})^{-\frac{1}{2}}\in \mathbb{K}(\mathcal{H})\) for all \(a\in \mathcal{A}\), such as for localizations as in \Cref{def:localization_horst}, we obtain that 
\begin{equation}
    [P_{\Ker (\mathcal{D}_-)}] = [P_{\Ker (\mathcal{D}_+)}]=0
\end{equation} 
since the Busby invariant constructed from these projections is zero using that \(a P_{\Ker \mathcal{D}} = a (1+\mathcal{D}^*\mathcal{D})^{-\frac{1}{2}} P_{\Ker \mathcal{D}} \in \mathbb{K}(\mathcal{H})\) for any \(a\in A\).
\end{remark}

\begin{remark}
    In \Cref{thm:boundary_map_of_even_horst}, the assumption that \(\ran \mathcal{D}_-\) (or \(\ran \mathcal{D}_+\)) is closed can be replaced with
    \begin{equation}
        a(1+\mathcal{D}^*\mathcal{D})^{-\frac{1}{2}}(1-P_{\Ker \mathcal{D}})\in \mathbb{K}(\mathcal{H}) \quad\text{and}\quad a(1+\mathcal{D}\mathcal{D}^*)^{-\frac{1}{2}}(1-P_{\Ker \mathcal{D}^*})\in \mathbb{K}(\mathcal{H})
    \end{equation}
    for any \(a\in A\) and still give the same result. This is similar to what was done in \cite[Proposition 2.15]{Forsyth_Goffeng_Mesland_Rennie_2019}.
\end{remark}

Returning to the geometric setting, if \(\Omega\) is an open set in a smooth manifold, then there is an exact sequence
\begin{equation}
    \label{eq:boundary_short_exact_seq}
    \begin{tikzcd}
    0 \ar[r] & C_0(\Omega) \ar[r] & C_0(\overline{\Omega}) \ar[r, "\gamma_0"] & C_0(\partial \Omega) \ar[r] & 0
    \end{tikzcd}
\end{equation}
where \(\gamma_0(a)=a|_{\partial \Omega}\). One can easily construct a completely positive right-inverse of \(\gamma_0\), which is also guaranteed by \cite[Theorem 3.3.6]{Higson_Roe_2000} since commutative \ensuremath{C^*}-algebras are nuclear \cite[Example 3.3.3]{Higson_Roe_2000}.

\begin{corollary}
    \label{thm:boundary_of_diff_op_is_PkerDmax}
    Let \(\Omega\subseteq M\) be an open set in a smooth manifold \(M\) and \(D\in \Diff(M;E,F)\) as in \Cref{ex:construct_formally_selfadjoint_diff_op}. If \(\ran D^\dagger_{\min}\) is closed, as for example when \(\overline{\Omega}\) is compact (see \Cref{remark:finite_dim_kernel}), then
    \begin{equation}
        \partial [D] = [P_{\Ker D_{\max}}] \quad\text{in}\quad K^1(C_0(\partial \Omega))
    \end{equation}
    where the action of \(C_0(\partial \Omega)\) on \(L^2(\Omega;E\oplus F)\) is point-wise multiplication via any \(*\)-linear right-inverse \(s\colon C_0(\partial \Omega)\to C_0(\overline{\Omega})\) of \(\gamma_0\).
\end{corollary}

\begin{remark}
    \label{remark:viewing_kerDmax_as_generalized_bergman}
    Let \(\mathcal{O}_D(\Omega)=\setcond{f\in \mathcal{D}(\Omega)}{Df=0}\), then \(\Ker D_{\max} = \mathcal{O}_D(\Omega)\cap L^2(\Omega;E)\) can be seen as a generalization of the Bergman space
    \(\mathcal{O}_{\overline{\partial}}(\Omega)\cap L^2(\Omega)\) constructed from \(\overline{\partial}\) in the case when \(\Omega\) is an open set in a complex manifold and \(\mathcal{O}_{\overline{\partial}}(\Omega)\) are the functions that are holomorphic in \(\Omega\).
\end{remark}

\begin{example}
    Given a higher-order spectral triple, one may ask if its \(K\)-homology class is the boundary map applied to the \(K\)-homology class of some higher-order relative spectral triple. This is indeed the case. Let \((\mathcal{A}, \mathcal{H}, \mathcal{D})\) be an odd higher-order spectral triple for \(A\) of order \(m\), then 
    \begin{equation}
        \partial \left[\begin{pmatrix}0 & (-\partial_t)_{\min} + \mathcal{D} \\ (\partial_t)_{\min} + \mathcal{D} & 0\end{pmatrix}\right] = [\mathcal{D}]
    \end{equation}
    where the boundary map acts as \(\partial\colon K^0(C_0((0,\infty)\otimes A)\to K^1(A)\) and the preimage is the exterior Kasparov product \([(-i\partial_t)_{\min}]\hat{\otimes} [\mathcal{D}]\) constructed as in \Cref{example:odd_odd_tensor_product_of_horst}. This is shown by the functorality of the boundary map and the exterior Kasparov product together with \(\partial [(-i\partial_t)_{\min}] = \mathds{1}_{\mathbb{C}}\).
\end{example}

\subsection{Independence of extension}
\label{sec:independence_of_extension}

Let \(\mathcal{D}\) be as in a higher-order relative spectral triple for \(J\vartriangleleft A\) and \(\mathcal{D}_e\) be an extension of \(\mathcal{D}\) satisfying \(\mathcal{D}\subseteq \mathcal{D}_e \subseteq \mathcal{D}^*\). Using \Cref{thm:F_is_almost_selfadjoint} we see that \(\mathcal{D}_e(1+\mathcal{D}_e^*\mathcal{D}_e)^{-\frac{1}{2}}\) is a \(J\)-compact perturbation of \(\mathcal{D}(1+\mathcal{D}^*\mathcal{D})^{-\frac{1}{2}}\).

\begin{corollary}[\(K\)-cycle is independent of extension]
\label{thm:Kcycle_independent_of_extension}
Let \((\mathcal{J}\vartriangleleft\mathcal{A},\mathcal{H},\mathcal{D})\) be a higher-order relative spectral triple for \(J\vartriangleleft A\) and \(\mathcal{D}_e\) be any extension of \(\mathcal{D}\) satisfying \(\mathcal{D}\subseteq \mathcal{D}_e \subseteq \mathcal{D}^*\). Then the bounded transform of \(\mathcal{D}_e\) defines a \(K\)-cycle for \(J\) (not relative) and 
\begin{equation}
    [(\rho, \mathcal{H}, \mathcal{D}_e(1+\mathcal{D}_e^*\mathcal{D}_e)^{-\frac{1}{2}})] = [(\rho, \mathcal{H}, \mathcal{D}(1+\mathcal{D}^*\mathcal{D})^{-\frac{1}{2}})] \quad\text{in}\quad  K^*(J).
\end{equation}
\end{corollary}

In particular, since \(K^*(J)\) is isomorphic to \(K^*(J\vartriangleleft A)\) by excision, the classes in \Cref{thm:Kcycle_independent_of_extension} are mapped to the same element by the boundary map described in \Cref{sec:boundary_map}. Hence, if there would exist an extension \(\mathcal{D}_e\) such that \((\rho, \mathcal{H}, \mathcal{D}_e(1+\mathcal{D}_e^*\mathcal{D}_e)^{-\frac{1}{2}})\) defines a \(K\)-cycle not only for \(J\) but also for \(A\), then \(\partial [\mathcal{D}]=0\). In this sense, \(\partial [\mathcal{D}]\) is an obstruction of well-behaved extensions. For a differential operator \(D\) on an open set \(\Omega\) that is elliptic in some calculus presented in \Cref{sec:host_examples}, such extensions are closed realizations, and realizations of \(D\) corresponds to fixing a boundary condition for \(D\) in the sense of \cite{Bandara_Goffeng_Saratchandran_2023}. Hence, any choice of boundary condition of \(D\) gives rise to the same class in \(K^*(C_0(\Omega))\) and the boundary map is an obstruction of the existence of boundary conditions that give \(K\)-cycles that can be lifted to \(K^*(C_0(\overline{\Omega}))\). This further motivates our notation \([D]\) from \Cref{ex:construct_formally_selfadjoint_diff_op}. Usage of this type of independence of boundary condition can be found \cite{Baum_van_Erp_2021}, where they use a specific boundary condition of \(\overline{\partial} + \overline{\partial}^*\) for a \(K\)-theoretic proof of Boutet de Monvel's index theorem for Toeplitz operators on pseudoconvex domains in \(\mathbb{C}^n\). 

\subsection{Independence of lower order perturbations}

This section will show that the \(K\)-homology class of a higher-order relative spectral triple constructed from a differential operator in an appropriate calculus (see \Cref{sec:host_examples}) does not change if we perturb the differential operator by something lower order in said calculus. Hence, the \(K\)-homology class only depends on the symbol of the differential operator in that calculus.

\begin{lemma}
    \label{thm:abstract_independence_of_lot}
    Let \(T\) be a closed densely defined operator such that \((\rho, \mathcal{H}, T(1+T^*T)^{-\frac{1}{2}})\) defines a \(K\)-cycle for a \ensuremath{C^*}-algebra \(J\) and \(j(1+T^*T)^{-\frac{1}{2}}\in \mathbb{K}(\mathcal{H})\) for all \(j\in J\). Let \(S\) be a closed densely defined operator on \(\mathcal{H}\) such that \(\dom S = \dom T\) and \((T-S)(1+T^*T)^{-\frac{1}{2}+\frac{1}{2m}}\) and \((T-S)(1+S^*S)^{-\frac{1}{2}+\frac{1}{2m}}\) are bounded where they are defined for some \(m>0\). Then \((\rho, \mathcal{H}, S(1+S^*S)^{-\frac{1}{2}})\) defines a \(K\)-cycle for \(J\) and 
    \begin{equation}
        [(\rho, \mathcal{H}, S(1+S^*S)^{-\frac{1}{2}})] = [(\rho, \mathcal{H}, T(1+T^*T)^{-\frac{1}{2}})] \quad\text{in}\quad K^*(J).
    \end{equation}
\end{lemma}

\begin{proof}

    It is enough to show that \((T(1+T^*T)^{-\frac{1}{2}} - S(1+S^*S)^{-\frac{1}{2}})j\in \mathbb{K}(\mathcal{H})\) for all \(j\in J\). This will be proved similarly to \Cref{thm:F_is_almost_selfadjoint} by using the integral formula \eqref{eq:integral_formula}. Namely,
    \begin{equation}
        (T(1+T^*T)^{-\frac{1}{2}} - S(1+S^*S)^{-\frac{1}{2}})j = \frac{1}{\pi} \int_1^\infty (\mu-1)^{-\frac{1}{2}} R_\mu ~d\mu
    \end{equation}
    point-wise on \(\dom T = \dom S\), where
    \begin{equation}
        R_\mu \coloneqq  (T(\mu+T^*T)^{-1} - S(\mu + S^*S)^{-1})j.
    \end{equation}
    We can rewrite \(R_\mu\) as
    \begin{equation}
        \begin{split}
            R_\mu
            =& (\mu + SS^*)(\mu + SS^*)^{-1} T(\mu + T^*T)^{-1} j \\
            &- S(\mu + S^*S)^{-1} (\mu + T^*T)(\mu + T^*T)^{-1}j \\
            =& \mu (\mu + SS^*)^{-1}(T-S)(\mu + T^*T)^{-1}j \\
            &- S(\mu + S^*S)^{-\frac{1}{2}}\left( (T - S) (\mu + S^*S)^{-\frac{1}{2}} \right)^* T (\mu + T^*T)^{-1}j
        \end{split}
    \end{equation}
    and using that \((\mu + T^*T)^{-\frac{1}{2}}j\in\mathbb{K}(\mathcal{H})\) and similar norm estimates as in \Cref{thm:F_is_almost_selfadjoint} we see that the integral is norm convergent over compact operators.
\end{proof}

\begin{corollary}
    \label{thm:independent_of_lot}
    Let \(\Omega\subseteq M\) be an open set in a smooth manifold \(M\) and \(D_1, D_2\) be formally self-adjoint differential operators of order \(m>0\) in one of the calculi presetned in \Cref{sec:host_examples} such that they have the same principal symbol in that calculus, then
    \begin{equation}
        [D_{1, \min}] = [D_{2, \min}] \quad\text{in}\quad K^*(C_0(\Omega)\vartriangleleft C_0(\overline{\Omega})).
    \end{equation}
\end{corollary}

Note that in the proof of \Cref{thm:abstract_independence_of_lot} the domains of \(T\) and \(S\) have to be treated carefully. In particular, we do not assume any relation between \(\dom T^*\) and \(\dom S^*\), as such domains correspond to the maximal domain of a differential operator in \Cref{thm:independent_of_lot}. The maximal domain of a differential operator can depend on lower-order terms as they affect the maximal kernel.

\Cref{thm:abstract_independence_of_lot} also hints at a way to handle differential operators with non-smooth coefficients.

\section{The even boundary map of a differential operator}
\label{sec:even_boundary_map_geometric_setting}

This section will consider the boundary map when applied to a higher-order relative spectral triple constructed from a classically elliptic differential operator on a compact smooth manifold with boundary. With the added assumption of smooth boundary we can refer to the classical results of Seeley in order to describe the image of the boundary map in boundary data only. This culminates in proving \Cref{thm:boundary_index_thm}.

\subsection{The trace operator}

This subsection will fix some notation for smooth manifolds with boundary and Lions-Magenes theorem. Let \(\Omega\subseteq M\) be a smooth manifold with boundary inside a smooth manifold \(M\). For vector bundles \(E\to \Omega\), we will abuse notation and write \(E\) for \(E|_{\partial \Omega}\to \partial \Omega\) when the base space is clear from the context. Fix a smooth inward-pointing normal \(\frac{\partial}{\partial x_n}\) of \(\partial \Omega\) and a cotangent section \(dx_n\) dual to \(\frac{\partial}{\partial x_n}\) near \(\partial \Omega\). Define the trace of order \(k\) as 
\begin{equation}
    \gamma_k\colon C^\infty(\overline{\Omega})\to C^\infty(\partial \Omega), f\mapsto \left.\left(\frac{\partial^k}{\partial x_n^k} f\right)\right|_{\partial \Omega}.
\end{equation}
We extend \(\gamma_k\) to sections of a vector bundle \(E\)
\begin{equation}
    \gamma_k\colon C^\infty(\overline{\Omega}; E)\to C^\infty(\partial \Omega; E),
\end{equation}
by trivializing to local charts. Define the \emph{full trace} of order \(m\geq 1\) as
\begin{equation}
    \gamma=\begin{pmatrix}\gamma_0 \\ \vdots \\ \gamma_{m-1}\end{pmatrix} \colon C^\infty(\overline{\Omega}; E)\to \bigoplus^m C^\infty(\partial \Omega; E).
\end{equation}
By the trace theorem \cite[Proposition 4.5, Chapter 4]{Taylor_2011} \(\gamma\) extends to a continuous surjection
\begin{equation}
    \label{eq:trace_theorem}
    \gamma\colon H^s(\overline{\Omega};E)\to \mathbb{H}^{s-\frac{1}{2}}(\partial \Omega;E\otimes \mathbb{C}^m)
\end{equation} 
for any \(s\in \mathbb{R}\) such that \(s>m-\frac{1}{2}\), where 
\begin{equation}
    \mathbb{H}^s(\partial \Omega; E\otimes \mathbb{C}^m) \coloneqq  \bigoplus_{0\leq l \leq m-1} H^{s-l}(\partial \Omega; E)
\end{equation}
is a \emph{mixed order Sobolev space}. Also, let 
\begin{equation}
    \mathbb{H}^s(\partial \Omega; E\otimes \mathbb{C}^m_ {\operatorname{op}} ) \coloneqq  \bigoplus_{0\leq l \leq m-1} H^{s+l}(\partial \Omega; E)
\end{equation}
which is the dual to \(\mathbb{H}^{-s}(\partial \Omega; E\otimes \mathbb{C}^m)\) in the \(L^2\)-pairing.

For the maximal domain of a classically elliptic differential operator which is a subset of \(L^2(\Omega;E)\), the trace theorem does not give a precise enough statement of regularity at the boundary. This will be remedied by the Lions-Magenes theorem.

Using the inward-pointing normal \(\frac{\partial}{\partial x_n}\) of \(\partial \Omega\) we can introduce coordinates close to the boundary \((x', x_n)\in \partial \Omega \times [0, 1) \subseteq \Omega\). A differential operator \(D\in \Diff(\Omega;E,F)\) can be written as
\begin{equation}
    \label{eq:diff_op_close_to_boundary}
    D = \sum_{j=0}^m A_j D_{x_n}^{m-j}
\end{equation}
close to the boundary where \(D_{x_n}=-i\frac{\partial}{\partial x_n}\) and \(A_j = (A_j(x_n))\) is a family of differential operators \(A_j(x_n)\in \Diff(\partial \Omega; E, F)\).

\begin{lemma}
    \label{thm:seeley_a}
    Let \(\Omega\subseteq M\) be a smooth manifold with boundary and \(D\in \Diff(M;E,F)\) be an operator of order \(m\). There is a matrix of differential operators \(\mathfrak{A}\colon C^\infty(\partial \Omega;E\otimes \mathbb{C}^m)\to C^\infty(\partial \Omega;F\otimes \mathbb{C}^m)\) on the form 
    \begin{equation}
        \mathfrak{A} = -i \begin{pmatrix}
            A_{m-1, 1} & A_{m-2, 1} & \cdots & A_{2,1} & A_{1,1} & A_0 \\ 
            A_{m-2, 2} & A_{m-3, 2} & \cdots & A_{1,2} & A_0 & 0 \\
            A_{m-3, 3} & &  & A_0 & 0 \\
            \vdots & & \iddots & & \vdots  \\
            A_{1, m-1} & A_0 & & & & \\
            A_0 & 0 & \cdots & & & 0
        \end{pmatrix} 
    \end{equation}
    satisfying
    \begin{equation}
        \label{eq:seeleya_formula}
        \braket{f, D^\dagger g}_{L^2(\Omega; E)} - \braket{D f, g}_{L^2(\Omega; F)} = \braket{\mathfrak{A}\gamma f, \gamma g}_{L^2(\partial \Omega; E\otimes \mathbb{C}^m)}
    \end{equation}
    for any \(f\in H^m(\overline{\Omega};E)\) and \(g\in H^m(\overline{\Omega};F)\), where \(A_{j,k}\in \Diff(\partial \Omega; E, F)\) is of order at most \(j\) and \(A_0=A_0|_{x_n=0}\in C^\infty(\partial \Omega; \Hom(E,F))\) is of order \(0\) as in \eqref{eq:diff_op_close_to_boundary}.
    
    Furthermore, if \(D\) is classically elliptic, then \(\mathfrak{A}\) is invertible and the inverse is also a matrix of differential operators. In particular, for any \(s\in \mathbb{R}\)
    \begin{equation}
        \mathfrak{A}\colon \mathbb{H}^s(\partial \Omega; E\otimes \mathbb{C}^m)\to \mathbb{H}^{s - (m - 1)}(\partial \Omega; E\otimes \mathbb{C}^m_ {\operatorname{op}} )
    \end{equation}
    is an isomorphism of Hilbert spaces.
\end{lemma}

\begin{proof}

    One constructs \(\mathfrak{A}\) from \eqref{eq:diff_op_close_to_boundary} using integration by parts and obtain \eqref{eq:seeleya_formula} for smooth functions. We can then extend \eqref{eq:seeleya_formula} to functions in \(H^m(\overline{\Omega};E)\) and \(H^m(\overline{\Omega};F)\) by continuity since the maps \(\mathfrak{A}\gamma\colon H^m(\overline{\Omega};E)\to \mathbb{H}^{\frac{1}{2}}(\partial \Omega; F\otimes \mathbb{C}^m_ {\operatorname{op}} )\) and \(\gamma:H^m(\overline{\Omega};F)\to \mathbb{H}^{-\frac{1}{2}}(\partial \Omega; F\otimes \mathbb{C}^m)\) are continuous and \(\mathbb{H}^{\frac{1}{2}}(\partial \Omega; F\otimes \mathbb{C}^m_ {\operatorname{op}} )\) is the dual of \(\mathbb{H}^{-\frac{1}{2}}(\partial \Omega; F\otimes \mathbb{C}^m)\) in the \(L^2\)-pairing. 
    
    Assuming that \(D\) is elliptic, we see that \(A_0=\sigma_0(A_0)=\sigma^m(D)(dx_n)\) on \(\partial \Omega\) is invertible. So the skew-diagonal in \(\mathfrak{A}\) is invertible, and we can algebraically obtain a matrix inverse with entries as polynomials in the entries of \(\mathfrak{A}\) multiplied by \((A_0)^{-1}\). 

\end{proof}

The construction of \(\mathfrak{A}\) in \Cref{thm:seeley_a} can be found in \cite[p. 794]{Seeley_1966}, \cite[Proposition 1.3.2]{Grubb_1996} and \cite[Proposition 2.15]{Bandara_Goffeng_Saratchandran_2023}.

\begin{theorem}[Lions-Magenes theorem]
    \label{thm:Lions_Magenes}
Let \(\overline{\Omega}\subseteq M\) be a compact smooth manifold with boundary and \(D\in \Diff(M;E,F)\) be a classically elliptic differential operator of order \(m\), then \(\gamma\) extends to a continuous map
\begin{equation}
    \label{eq:Lions_Magenes}
    \gamma\colon\dom D_{\max}\to \mathbb{H}^{-\frac{1}{2}}(\partial \Omega;E\otimes \mathbb{C}^m)
\end{equation} 
with dense range and \(\Ker \gamma = \dom D_{\min}=H_0^m(\Omega;E)\). 
\end{theorem}

\begin{proof}

    Let \(\mathfrak{A}\) be as in \Cref{thm:seeley_a}. Note that \(\mathfrak{A}^\dagger\) has the same form as \(\mathfrak{A}\) and is also invertible. By the trace theorem in \eqref{eq:trace_theorem} \(\gamma\colon H^m(\partial \Omega; F)\to\mathbb{H}^{m-\frac{1}{2}}(\partial \Omega; F\otimes \mathbb{C}^m)\) is a continuous surjection, and hence there is a linear right inverse \(\mathcal{E}\colon\mathbb{H}^{m-\frac{1}{2}}(\partial \Omega; F\otimes \mathbb{C}^m)\to H^m(\overline{\Omega}; F)\) of \(\gamma\) which is continuous by the open mapping theorem \cite[Theorem 2.6]{Brezis_2011}. Using \Cref{thm:seeley_a} we obtain that both \(\mathcal{E} (\mathfrak{A}^\dagger)^{-1}\) and \(D^\dagger \mathcal{E} (\mathfrak{A}^\dagger)^{-1}\) are continuous maps \(\mathbb{H}^{\frac{1}{2}}(\partial \Omega; E\otimes \mathbb{C}^m_ {\operatorname{op}} )\to L^2(\Omega;F)\).

    By \Cref{thm:seeley_a}
    \begin{equation}
        \label{eq:lionsmagenes_boundary_pairing}
        \braket{\gamma f, v}_{L^2(\partial \Omega; E\otimes \mathbb{C}^m)} = \braket{D f, \mathcal{E} (\mathfrak{A}^\dagger)^{-1} v}_{L^2(\Omega; F)} - \braket{f, D^\dagger \mathcal{E} (\mathfrak{A}^\dagger)^{-1} v}_{L^2(\Omega; E)}
    \end{equation}
    for any \(f\in C^\infty(\overline{\Omega};E)\) and \(v\in C^\infty(\partial \Omega;E\otimes \mathbb{C}^m)\) since \(v=\mathfrak{A}^\dagger \gamma \mathcal{E} (\mathfrak{A}^\dagger)^{-1} v\) and \(\mathcal{E} (\mathfrak{A}^\dagger)^{-1} v\in H^m(\overline{\Omega};F)\). It follows from \eqref{eq:lionsmagenes_boundary_pairing} that \(\norm{\gamma f}_{\mathbb{H}^{-\frac{1}{2}}(\partial \Omega; E\otimes \mathbb{C}^m)}\lesssim \norm{f}_D\) for any \(f\in C^\infty(\overline{\Omega};E)\) since \(\mathbb{H}^{\frac{1}{2}}(\partial \Omega; E\otimes \mathbb{C}^m_ {\operatorname{op}} )\) is the dual of \(\mathbb{H}^{-\frac{1}{2}}(\partial \Omega; E\otimes \mathbb{C}^m)\) in the \(L^2\)-pairing. We have that \(C^\infty(\overline{\Omega};E)\subseteq \dom D_{\max}\) is dense by \Cref{remark:smooth_boundary_has_smooth_approximation} and \Cref{thm:smooth_approximation}, hence we can extend \(\gamma\) to \(\dom D_{\max}\).

    Now that we have shown that \eqref{eq:Lions_Magenes} is continuous, we can extend \eqref{eq:seeleya_formula} to \(f\in \dom D_{\max}\) and \(g\in H^m(\overline{\Omega};F)\) in the same way as in the proof of \Cref{thm:seeley_a}. In particular, for any \(f\in \dom D_{\max}\) we see that \(\braket{D f, g}_{L^2(\Omega; F)} = \braket{f, D^\dagger g}_{L^2(\Omega; E)}\) for any \(g\in H^m(\overline{\Omega};F)\) which implies that \(\Ker\gamma \subseteq \dom (D^\dagger_{\max})^*=\dom D_{\min}\). The opposite inclusion follows from density of \(C_c^\infty(\Omega;E)\) in \(\dom D_{\min}\).

    That the range is dense follows from that \(C^\infty(\partial \Omega;E\otimes \mathbb{C}^m)= \gamma(C^\infty(\overline{\Omega};E))\).
\end{proof}

This is a classical result from \cite{Lions_Magenes_1972} and is also proven in \cite{Seeley_1966}. Note that we need to use the specific property of the classical Sobolev spaces in \Cref{thm:smooth_approximation}.

\subsection{The Caldéron projector}

This subsection will present the Calderón projector constructed by Seeley, which is an idempotent with range \(H_D\coloneqq\gamma(\Ker D_{\max}) \subseteq \mathbb{H}^{-\frac{1}{2}}(\partial \Omega; E\otimes \mathbb{C}^m)\). Moreover, the Calderón projector is a pseudo-differential operator in the Douglis-Nirenberg calculus which is essential for our later use in relating it to the class \([P_{\Ker D_{\max}}]\in K^1(C(\partial \Omega))\) that appeared in \Cref{thm:boundary_of_diff_op_is_PkerDmax}. We will also present Hörmander's symbol calculation for the Calderón projector.

\begin{definition}[Douglis-Nirenberg calculus]
    Let \(Y\) be a closed smooth manifold and \(E\to Y\) a Hermitian vector bundle. The set of pseudo-differential operators of order \(n\) in the \emph{Douglis-Nirenberg calculus} \(\Psi_{\textup{DN}}^n(Y;E\otimes \mathbb{C}^m)\) consists of matrices \(T=(T_{jk})\) of operators \(T_{jk}\in \Psi_{\textup{cl}}^{n + j - k}(Y; E)\). We see that
    \begin{equation}
        T\colon \mathbb{H}^s(Y; E\otimes \mathbb{C}^m)\to \mathbb{H}^{s-n}(\partial \Omega; E\otimes \mathbb{C}^m)
    \end{equation}
    is continuous for any \(s\in \mathbb{R}\). The principal symbol of \(\sigma_{\textup{DN}}^n(T)\) of \(T\in \Psi_{\textup{DN}}^n(Y;E)\) is defined as the matrix \(\sigma^m_{\textup{DN}}(T)_{jk} = \sigma^{m + j - k}(T_{jk})\). Hence, one can view it as \(\sigma^m_{\textup{DN}}(T)\in C^\infty(S^*\partial \Omega; M_m(\End(\varphi^*(E))))\) where \(\varphi\colon S^*\partial \Omega\to \partial \Omega\) is the fiber map. 
    
    See \cite[Chapter XIX.5]{Hormander_2007} for more on the Douglis-Nirenberg calculus, where it is defined in greater generality.
\end{definition}

\begin{theorem}
\label{thm:seeleys}
    Let \(\Omega\) be a compact smooth manifold with boundary and \(D\in \Diff(\Omega; \newline E, F)\) be a classically elliptic differential operator of order \(m\). There is an idempotent \(P_C \in \mathbb{B}(\mathbb{H}^{-\frac{1}{2}}(\partial \Omega; E\otimes \mathbb{C}^m))\) with range \(H_D\) such that \(P_C\in \Psi_{\textup{DN}}^0(\partial \Omega; E\otimes \mathbb{C}^m)\) \cite{Seeley_1966}.
\end{theorem}

We refer to \(P_C\) constructed in \cite{Seeley_1966} as \emph{the} Calderón projector. Note that \Cref{thm:seeleys} ensures that \(H_D\) is a closed subspace of \(\mathbb{H}^{-\frac{1}{2}}(\partial \Omega; E\otimes \mathbb{C}^m)\). We call \(H_D\) the \emph{Hardy space}, terminology taken from \cite{Bandara_Goffeng_Saratchandran_2023}.

The Calderón projector is constructed as \(P_C=\gamma K\) where \(K\in \mathbb{B}(\mathbb{H}^{-\frac{1}{2}}(\partial \Omega; E\otimes \mathbb{C}^m), L^2(\Omega;E)\) is a Poisson operator with range \(\Ker D_{\max}\) such that \(K\gamma f - f \in \Ker D_{\min}\) for any \(f\in \Ker D_{\max}\) (see \cite[Example 1.3.5]{Grubb_1996}). In turn, \(K=r_{\Omega} Q \gamma^\dagger \mathfrak{A}\) where \(r_\Omega\colon L^2(M;E)\to L^2(\Omega;E)\) is the restriction from the closed manifold \(M\) that \(\Omega\) is a domain in, \(Q\) is a certain choice of parametrix of \(D\) on \(M\), \(\gamma^\dagger\) is the distributional adjoint of \(\gamma\) from \(M\) to \(\partial \Omega\) and \(\mathfrak{A}\) that appeared in \Cref{thm:seeley_a} \cite[(8)]{Seeley_1966}, \cite[(1.3.21)]{Grubb_1996}. For later use of \(K\), we give the following lemma.

\begin{lemma}
    \label{thm:Poisson_commutator_property}
    For any \(a\in C^\infty(\partial \Omega)\),
    \begin{equation}
        Ka - s(a) K\in \mathbb{K}(\mathbb{H}^{-\frac{1}{2}}(\partial \Omega; E \otimes \mathbb{C}^m), L^2(\Omega;E))
    \end{equation}
    where \(s\colon C(\partial \Omega)\to C(\overline{\Omega})\) is any \(*\)-linear right-inverse of \(\gamma_0\) that conserves smoothness and extends functions to be constant in the orthogonal direction of the boundary close to the boundary (such an \(s\) can easily be constructed).
\end{lemma}

\begin{proof}
    Fix an \(a\in C^\infty(\partial \Omega)\). The choice of \(s\) ensures that \(\gamma_k(s(a)) = 0\) for \(k>0\) where we use \(\gamma_k\) as a map \(C^\infty(\overline{\Omega})\to C^\infty(\partial \Omega)\). Hence, \(a\gamma = \gamma s(a) \) acting on \(C^\infty(\overline{\Omega};E)\). Therefore, \(\gamma^\dagger a = s(a)\gamma^\dagger\) where both sides extends to maps \(\mathbb{H}^{-m+\frac{1}{2}}(\partial \Omega; E\otimes \mathbb{C}^m_ {\operatorname{op}} )\to H^{-m}(M;F)\). The statement now follows from looking at the construction of \(K\) and noting that \([\mathfrak{A}, a]\in \mathbb{K}(\mathbb{H}^{-\frac{1}{2}}(\partial \Omega; E\otimes \mathbb{C}^m), \mathbb{H}^{-m+\frac{1}{2}}(\partial \Omega; F\otimes \mathbb{C}^m_ {\operatorname{op}} ))\) and \([Q, s(a)]\in \mathbb{K}(H^{-m}(M;F), L^2(M;E))\).
\end{proof}

A similar construction to that of the Calderón projector can be found in \cite[Equation 20.1.7]{Hormander_2007} where an arbitrary parametrix is used, resulting in \(P_C\) up to some smoothing operator. In \cite{Hormander_2007}, an explicit description of the symbol of \(P_C\) is also given which we now will present.

Let \((x', x_n)\in \partial \Omega \times [0, 1) \subseteq \overline{\Omega}\) be coordinates close to the boundary coherent with \eqref{eq:diff_op_close_to_boundary} and decompose \(T^*\overline{\Omega}\) close to the boundary as \((x', x_n, \xi', \xi_n)\in T^* \partial \Omega \oplus \Span(dx_n)\subseteq T^* \overline{\Omega}\).

\begin{definition}
    \label{def:boundary_symbol}
Define the \emph{boundary symbol}
\begin{equation}
    \sigma_\partial (D) \colon  T^* \partial \Omega \to \Diff(\mathbb{R})\otimes \Hom(\varphi^*(E|_{\partial \Omega}), \varphi^*(F|_{\partial \Omega}))
\end{equation}
of \(D\in \Diff(\Omega; E, F)\) as point-wise being the ordinary differential operator in a free variable \(t\in \mathbb{R}\)
\begin{equation}
    \sigma_\partial (D) (x',\xi') = \sum a_j(x', \xi') D_t^{m-j}
\end{equation}
where \(a_j\) is the principal symbol of \(A_j|_{x_n=0}\) constructed in \eqref{eq:diff_op_close_to_boundary} and \(\varphi\colon T^*\partial \Omega\to \partial \Omega\) is the fiber map. Note that the free variable \(t\) is associated with the coordinate \(x_n\) close to the boundary, but they are not the same.
\end{definition}

\begin{lemma}
    \label{thm:boundary_symbol_diff_eq_has_solutions}
    Let \(\Omega\) be a compact smooth manifold with boundary and \(D\in \Diff(\Omega; \newline E, F)\) be an elliptic differential operator of order \(m\). For fixed \((x',\xi')\in S^*\partial \Omega\), the ordinary differential equation 
\begin{equation}
    \label{eq:homogenus_ode_of_boundary_symbol}
    \sigma_\partial(D)(x', \xi') v = 0
\end{equation}
has a unique solution \(v\in C^\infty(\mathbb{R}; \varphi^*(E|_{\partial \Omega})_{(x',\xi')})\) given any initial condition 
\begin{equation}
    \begin{pmatrix}v(0) \\ D_t v(0) \\ \vdots \\ D_t^{m-1} v(0)\end{pmatrix} \in \varphi^*(E|_{\partial \Omega})_{(x',\xi')}\otimes \mathbb{C}^m.
\end{equation}
Furthermore, the solution can be uniquely decomposed as \(v = v_+ + v_-\) where \(v_\pm\) decays exponentially as \(t\to \pm\infty\).
\end{lemma}

\begin{proof}
    Note that \(a_0=\sigma_m(D)(dx_n)\) is invertible. Therefore, we can rewrite \eqref{eq:homogenus_ode_of_boundary_symbol} as \(\partial_t V = A V\) where \(V=\begin{pmatrix}v & \partial_t v & \cdots & \partial_t^{m-1}v\end{pmatrix}^T\) which Picard-Lindlöf theorem guarantees has a unique solution for any initial condition.
    
    Furthermore, the solutions will be in the form \(e^{At}V_0\). These solutions contain exponentials with exponent corresponding to the eigenvalues of \(A\), which are exactly the zeros of the characteristic polynomial
    \begin{equation}
        \det \left(\sum a_j(x', \xi') (-iz)^{m-j}\right).
    \end{equation}
    This has no purely imaginary roots \(z=ir\) for \(r\in\mathbb{R}\), since then it is equal to 
    \begin{equation}
        \det(\sigma^m(D)(x',0,\xi', r))    
    \end{equation}
    which is non-zero by ellipticity. The solutions can therefore be decomposed as a sum of solutions exponentially decaying as either \(t\to +\infty\) or \(t\to -\infty\).
\end{proof}

\begin{definition}
    \label{def:calderon_bundle_decomposition}
    From the homogenous differential equations of the boundary symbol in \Cref{thm:boundary_symbol_diff_eq_has_solutions} we obtain a decomposition of vector bundles over \(S^*\partial \Omega\) as
    \begin{equation}
        \label{eq:calderon_bundle_decomposition}
        \varphi^*(E|_{\partial \Omega})\otimes \mathbb{C}^m = E_+(D) \oplus E_-(D)
    \end{equation}
    where \(\varphi\colon S^*\partial \Omega\to \partial \Omega\) is the fiber map and \(E_\pm(D)\) corresponds to initial conditions that give solutions decaying exponentially as \(t\to \pm\infty\). Note however that the decomposition \eqref{eq:calderon_bundle_decomposition} is not orthogonal in general.
\end{definition}

\begin{theorem}
    \label{thm:hormander_calderon_symbol}
    The principal symbol of the Calderón projector can be calculated as
    \begin{equation}
        \sigma^0_{\textup{DN}}(P_C)_{jk} = \sum_{\imag \xi_n >0} \Res_{\xi_n}\left[ \sum_{l=0}^{m-k-1} \xi_n^{j+l} a^{-1}(x',0,\xi',\xi_n)  a_{(m-k-1)-l}(x',0,\xi')  \right]
    \end{equation}
    where
    \begin{equation}
        \sigma^m(D) = a(x',x_n,\xi',\xi_n) = \sum_{l=0}^m a_l(x',x_n,\xi') \xi_n^{m-l}
    \end{equation}
    \cite[Equation 20.1.8]{Hormander_2007}.

    Moreover, \(\sigma^0_{\textup{DN}}(P_C)\) seen as an element in \(C^\infty(S^*\partial \Omega; M_m(\End(\varphi^*(E|_{\partial \Omega}))))\) is the projector onto \(E_+(D)\) along \(E_-(D)\) \cite[Theorem 20.1.3]{Hormander_2007}.
\end{theorem}

\begin{remark}
    For \(m>1\), note that \(P_C\) is not generally an orthogonal projection as an operator in \(\mathbb{B}(\mathbb{H}^{-\frac{1}{2}}(\partial \Omega; E\otimes \mathbb{C}^m))\). One can see this on the symbol since the \(E_+(D)\) and \(E_-(D)\) need not be orthogonal as subbundles of \(\varphi^*(E|_{\partial \Omega})\otimes \mathbb{C}^m\).
\end{remark}

\subsection{Order reduction}

When constructing an element in \(K^1(C(\partial \Omega))\), the natural representation of \(C(\partial \Omega)\) on some function space is multiplication. In order to have the adjoint of a multiplication operator to be complex conjugate, the function space needs to be some sort of \(L^2\)-space, while with \Cref{thm:Lions_Magenes} we instead land in \(\mathbb{H}^{-\frac{1}{2}}(\partial \Omega; E\otimes \mathbb{C}^m)\). This section will construct a Hilbert space isomorphism to remedy this issue.

Let \(\Lambda_\beta\in\Psi_{\textup{cl}}^\beta(\partial \Omega;E)\) be any invertible positive elliptic operator with symbol 
\begin{equation}
    \sigma^0(\Lambda_\beta)=|\xi'|_E^\beta \mathds{1}_{\End(\varphi^*(E))}.
\end{equation}
For example, one could construct these from the Bochner-Laplacian, that is, take any connection \(\nabla_E\colon C^\infty(\partial \Omega;E) \to C^\infty(\partial \Omega; T^*\partial \Omega \otimes E)\) and let \(\Lambda_2=1+\nabla_E^\dagger \nabla_E\) and \(\Lambda_\beta = (\Lambda_2)^{\frac{\beta}{2}}\) up to some smoothing operator. 

Now each
\begin{equation}
    \Lambda_\beta\colon H^s(\partial \Omega;E)\to H^{s-\beta}(\partial \Omega; E)
\end{equation}
is an isomorphism for any \(s\in \mathbb{R}\), which means that
\begin{equation}
    \lambda_\alpha = \diag_{0\leq j\leq m-1}( \Lambda_{\alpha+j} ) = \begin{pmatrix}\Lambda_{\alpha} & & \\ & \ddots & \\ & & \Lambda_{\alpha+m-1}\end{pmatrix}
\end{equation}
defines an isomorphism
\begin{equation}
    \lambda_\alpha \colon  H^s(\partial \Omega ; E \otimes \mathbb{C}^m) \to \mathbb{H}^{s-\alpha}(\partial \Omega ; E \otimes \mathbb{C}^m)
\end{equation}
for any \(s\in \mathbb{R}\). Note that \(\lambda_{-\alpha} \neq \lambda_\alpha^{-1}\).

\begin{lemma}
    \label{thm:orderreduction_is_pseudo_and_has_same_symbol}
    We have that \(\lambda_{\frac{1}{2}}^{-1}P_C \lambda_{\frac{1}{2}}\in \Psi_{\textup{cl}}^0(\partial \Omega; E\otimes \mathbb{C}^m)\) is an idempotent in \(\mathbb{B}(L^2(\partial \Omega; E\otimes \mathbb{C}^m))\) with \(\sigma^0_{\textup{DN}}(P_C) = \sigma^0(\lambda_{\frac{1}{2}}^{-1}P_C \lambda_{\frac{1}{2}})\) as projection valued elements in \(C^\infty(S^*\partial \Omega; M_m(\End(\varphi_*(E|_{\partial \Omega}))))\).
\end{lemma}

\begin{proof}
    This follows from that \(P_C\in \Psi_{\textup{DN}}^0(\partial \Omega; E\otimes \mathbb{C}^m)\) by \Cref{thm:seeleys} and that
    \begin{equation}
        \sigma^0(\lambda_{\frac{1}{2}}^{-1}P_C \lambda_{\frac{1}{2}})_{j,k} = |\xi'|_E^{-\frac{1}{2}-j}\sigma^{j-k}({P_C}_{j,k})|\xi'|_E^{\frac{1}{2}+k} = |\xi'|_E^{-j+k}\sigma^0_{\textup{DN}}(P_C)_{jk}
    \end{equation}
    in \(C^\infty(T^*\partial \Omega;\End(\varphi^*(E|_{\partial \Omega})))\).
\end{proof}

\subsection{Projections corresponding to projectors}

This short subsection will present tools for later use to remedy the fact that the Calderón projector is not a projection, only an idempotent. 

\begin{lemma}
    \label{thm:idempotent_and_projection}
    Let \(A\) be a \ensuremath{C^*}-algebra and \(e\in A\) be an idempotent. Let \begin{equation}
        \label{eq:kaplansky_formula}
        p=ee^*(1+(e-e^*)(e^*-e))^{-1},
    \end{equation}
    then \(p\) is a projection, \(ep=p\), \(pe=e\) and \([e]=[p]\in K_0(A)\).
\end{lemma}

This formula \eqref{eq:kaplansky_formula} is due to Kaplansky \cite[Theorem A]{Kaplansky_Berberian_1955} and a proof can also be found in \cite[Proposition 4.6.2]{Blackadar_1986}.

\begin{corollary}
    \label{thm:projection_is_pseudo_if_projector_is}
    For an idempotent \(P\in\mathbb{B}\), the orthogonal projection onto its range \(P_{\ran P}\) is equal to \(PP^*(1+(P-P^*)(P^*-P))^{-1}\).
    
    As a consequence, if \([a,P],[a,P^*]\in\mathbb{K}\) then \([a, P_{\ran P}]\in\mathbb{K}\). If furthermore \(P\in \Psi_{\textup{cl}}^0(Y;E)\), then so is \(P_{\ran P}\) and 
    \begin{equation}
        [\sigma^0(P)]=[\sigma^0(P_{\ran P})] \quad\text{in}\quad K_0(C(S^*Y)).
    \end{equation}
\end{corollary}

\subsection{A \texorpdfstring{\(K\)}{ K}-cycle from the Calderón projector}
\label{sec:Kcycle_from_calderon}

In this subsection, we construct a representative of the class \([P_{\Ker D_{\max}}]\in K^1(C(\partial \Omega))\) from \Cref{thm:boundary_of_diff_op_is_PkerDmax} in terms of boundary data. Since \(\ran P_C = H_D \coloneqq \gamma (\Ker D_{\max})\), we will do this using the Calderón projector and the trace operator. However, since \(P_C\) is not an orthogonal projection and \(H_D\) is not an \(L^2\)-space, we will not build a \(K\)-homology class from \(P_C\) directly.

\begin{lemma}
    \label{thm:boundary_data_kcycle_is_welldefined_and_symbol_ktheory}
    Let \(P_{\widehat{H}_D}\) be the orthogonal projection onto \(\widehat{H}_D\coloneqq \lambda_{\frac{1}{2}}^{-1}(H_D)\) in \(L^2(\partial \Omega; \newline E\otimes \mathbb{C}^m)\), then \(P_{\widehat{H}_D}\in \Psi_{\textup{cl}}^0(\partial \Omega; E\otimes \mathbb{C}^m)\). In particular, \([P_{\widehat{H}_D}] \in K^1(C(\partial \Omega))\) is well-defined with representation of \(C(\partial \Omega)\) on \(L^2(\partial \Omega; E\otimes \mathbb{C}^m)\) as point-wise multiplication.

    Furthermore, \([\sigma^0(P_{\widehat{H}_D})] = [E_+(D)]\) in \(K_0(C(S^*\partial \Omega))\).
\end{lemma}

\begin{proof}
    By \Cref{thm:orderreduction_is_pseudo_and_has_same_symbol}, \(\lambda_{\frac{1}{2}}^{-1}P_C \lambda_{\frac{1}{2}}\in \Psi_{\textup{cl}}^0(\partial \Omega; E\otimes \mathbb{C}^m)\) which has range \(\widehat{H}_D\) and symbol \(\sigma^0_{\textup{DN}}(P_C)\). Now \Cref{thm:projection_is_pseudo_if_projector_is} give us that \(P_{\widehat{H}_D}\in \Psi_{\textup{cl}}^0(\partial \Omega; E\otimes \mathbb{C}^m)\) and \([\sigma^0(P_{\widehat{H}_D})] = [\sigma^0_{\textup{DN}}(P_C)]\). Lastly, note that \([E_+(D)]\in K_0(C(S^*\partial \Omega))\) coincides with the class of any projection in \(M_N(C(S^*\partial \Omega))\), for some \(N\), with range isomorphic to \(E_+(D)\), which is fulfilled by \(\sigma^0_{\textup{DN}}(P_C)\) by \Cref{thm:hormander_calderon_symbol}.
\end{proof}

\begin{theorem}
\label{thm:boundary_data_kcycle_equality}
Let \(\overline{\Omega}\) be a compact smooth manifold with boundary and \(D\in \Diff(\Omega; \newline E,F)\) be a classically elliptic differential operator of order \(m\), then
\begin{equation}
    [P_{\Ker D_{\max}}] = [P_{\widehat{H}_D}] \quad\text{in}\quad  K^1(C(\partial \Omega)),
\end{equation}
where \(P_{\widehat{H}_D}\) is as in \Cref{thm:boundary_data_kcycle_is_welldefined_and_symbol_ktheory}.
\end{theorem}

To prove \Cref{thm:boundary_data_kcycle_equality} we will use \Cref{thm:similar_busby_invariants} with the operator \(\lambda_{\frac{1}{2}}^{-1}\gamma\colon \Ker D_{\max} \to \widehat{H}_D\). The following lemma shows the needed properties.

\begin{lemma}
    \label{thm:trace_map_order_reduction_is_fredholm_and_so_on}
    The map
    \begin{equation}
        \label{eq:trace_map_order_reduction_map}
        \lambda_{\frac{1}{2}}^{-1}\gamma\colon \Ker D_{\max} \to \widehat{H}_D
    \end{equation} 
    is Fredholm and for all \(a\in C(\partial \Omega)\),
    \begin{equation}
        P_{\widehat{H}_D} a \lambda_{\frac{1}{2}}^{-1}\gamma - \lambda_{\frac{1}{2}}^{-1}\gamma P_{\Ker D_{\max}} s(a) \in \mathbb{K}(\Ker D_{\max}, \widehat{H}_D)
    \end{equation}
    where \(s\colon C(\partial \Omega)\to C(\overline{\Omega})\) is any \(*\)-linear right-inverse of \(\gamma_0\).
\end{lemma}

\begin{proof}
    Note that \(\gamma\colon \Ker D_{\max} \to H_D\coloneqq \gamma(\Ker D_{\max})\) is surjective by definition and using \Cref{thm:Lions_Magenes} we see that
    \begin{equation}
        \Ker (\gamma \colon \dom D_{\max} \to \mathbb{H}^{-\frac{1}{2}}(\partial \Omega; E\otimes \mathbb{C}^m) )|_{\Ker D_{\max}} = \dom D_{\min} \cap \Ker D_{\max} = \Ker D_{\min}
    \end{equation}
    which is finite-dimensional (see \Cref{remark:finite_dim_kernel}). Since \(\lambda_{\frac{1}{2}}^{-1} \colon H_D \to \widehat{H}_D \coloneqq \lambda_{\frac{1}{2}}^{-1} (H_D)\) is an isomorphism we obtain Fredholmness of \eqref{eq:trace_map_order_reduction_map}.

    For the second statement, firstly note that by density we can restrict to look at \(a\in C^\infty(\partial \Omega)\) since \(\Ker D_{\max}\subseteq L^2(\Omega;E)\), \(\widehat{H}_D\subseteq L^2(\partial \Omega; E\otimes \mathbb{C}^m)\) and the operator norm of \(a\) and \(s(a)\) acting on \(L^2\)-spaces is the supremum norm. Secondly note that we can choose \(s\colon C(\partial \Omega)\to C(\overline{\Omega})\) freely since \(P_{\Ker D_{\max}} b P_{\Ker D_{\max}}\in \mathbb{K}(L^2(\Omega;E))\) for any \(b\in C_0(\Omega)\). We choose an \(s\) satisfying the property needed in \Cref{thm:Poisson_commutator_property}. 
    
    Fix an \(a\in C^\infty(\partial \Omega)\). Since \([\lambda_{\frac{1}{2}}^{-1}, a] \in \mathbb{K}(\mathbb{H}^{-\frac{1}{2}}(\partial \Omega; E \otimes \mathbb{C}^m), L^2(\partial \Omega; E \otimes \mathbb{C}^m))\), \([P_C, a] \in \mathbb{K}(\mathbb{H}^{-\frac{1}{2}}(\partial \Omega; E \otimes \mathbb{C}^m))\) and \(P_{\widehat{H}_D} = \lambda_{\frac{1}{2}}^{-1} P_{H_D} \lambda_{\frac{1}{2}}\) and \(\ran P_C = H_D\) we have that 
    \begin{equation}
        P_{\widehat{H}_D} a \lambda_{\frac{1}{2}}^{-1} \gamma - \lambda_{\frac{1}{2}}^{-1} P_C a \gamma \in \mathbb{K}(\Ker D_{\max}, \widehat{H}_D).
    \end{equation}
    Using the construction of \(P_C\) as \(P_C=\gamma K\) where \(\ran K = \Ker D_{\max}\) we have that 
    \begin{equation}
        P_C a \gamma = \gamma K a \gamma = \gamma P_{\Ker D_{\max}}K a \gamma
    \end{equation}
    on \(\Ker D_{\max}\). Using \Cref{thm:Poisson_commutator_property} and that \(K\gamma - 1\in \mathbb{K}(\Ker D_{\max})\) since it has finite-dimensional range \(\Ker D_{\min}\), we get that
    \begin{equation}
        K a \gamma - P_{\Ker D_{\max}} s(a) P_{\Ker D_{\max}} \in \mathbb{K}(\Ker D_{\max}),
    \end{equation}
    and we are done.
\end{proof}

\begin{remark}
    In the first-order case, the proof of \Cref{thm:trace_map_order_reduction_is_fredholm_and_so_on} could be made less convoluted since then \(\gamma s(a)\) can be extended to act on \(\Ker D_{\max}\). This is not the case for higher orders since then \(C^\infty(\overline{\Omega})\) do not necessarily preserve \(\Ker D_{\max}\) (see \Cref{remark:a_domDadj_notin_domD}) and \(\gamma\) is not defined on \(L^2(\Omega;E)\).
\end{remark}

\subsection{The cap product}
\label{sec:cap_product}

This subsection will further dismantle the \(K\)-cycle \([P_{\widehat{H}_D}]\in K^1(C(\partial \Omega))\) in \Cref{thm:boundary_data_kcycle_is_welldefined_and_symbol_ktheory} with the use of the cap product in order to conclude the proof of \Cref{thm:boundary_index_thm}.

\begin{definition}
For a commutative, separable \ensuremath{C^*}-algebra \(A\) define the \emph{cap product} 
\begin{align}
    K_p(A)\times K^q(A) &\to K^{p+q}(A) \\
    (x,y) &\mapsto x\cap y \coloneqq  (x\hat{\otimes} \mathds{1}_A)\otimes_{A\otimes A} [\diag^*] \otimes_{A} y
\end{align}
using the identifications \(K_p(A)\cong KK^p(\mathbb{C},A)\) and \(K^q(A)\cong KK^q(A,\mathbb{C})\), where \(\hat{\otimes}\) and \(\otimes_B\) denotes is the exterior and interior Kasparov product respectively and \(\diag^*\colon A\otimes A\to A, ~a_1\otimes a_2 \mapsto a_1a_2\). Note that is a construction exclusively for commutative \ensuremath{C^*}-algebras, as \(\diag^*\) comes from \(\diag\colon X \to X\times X, ~ x \to (x,x)\) when \(A=C_0(X)\).
\end{definition}

In our case, we will look at \(p=0\) and \(q=1\).

\begin{lemma}
    \label{thm:cap_prod_calculation}
    Let \(A\) be a commutative, separable, trivially graded \ensuremath{C^*}-algebra, \([p]\in K^0(A)\) for a projection \(p\in M_N(A)\) and \([\beta]\in K^1(A)\) for a Busby invariant \(\beta\colon A\to\mathcal{Q}(\mathcal{H})\), then
    \begin{equation}
        [p]\cap [\beta] = [\beta_p] \quad\text{in}\quad K^1(A),
    \end{equation}
    where \(\beta_p\colon A\to\mathcal{Q}(\mathcal{H}^N),~a\mapsto \beta(pap)\).
\end{lemma}

The proof is well known and can be found in for instance in \cite[Corollary 4.4]{Gerontogiannis_2022}.

On a closed \(\text{spin}^c\)~ manifold \(Z\), one can construct a Dirac operator \(\slashed{D}\) given a fixed Riemannian metric. From \(\slashed{D}\) we can construct a spectral triple for \(C(Z)\) and hence a class in \(K^*(C(Z))\), where the parity matches the parity of the dimension of \(Z\). This class is independent of the choice of metric used to construct \(\slashed{D}\) \cite[Definition 11.2.10]{Higson_Roe_2000} and generates \(K^*(C(Z))\) as a \(K_*(C(Z))\)-module using the cap product \cite[Exercise 11.8.11]{Higson_Roe_2000}. Denote class constructed from \(\slashed{D}\) as \([Z]\in K^*(C(Z))\) and call it the fundamental class of \(Z\). 

For a closed smooth manifold \(Y\), we have that \(S^*Y\) is an odd-dimensional closed \(\text{spin}^c\)~ manifold. The fundamental class \([S^*Y]\) is equal to the class from the Toeplitz extension
\begin{equation}
    \begin{tikzcd}
        0 \ar[r] & \mathbb{K}(H^2(B^*Y)) \ar[r] & \mathcal{T}(S^*Y) \ar[r] & C(S^*Y) \ar[r] & 0
    \end{tikzcd}
\end{equation}
by \cite{Baum_van_Erp_2021}, which in turn is equal to the class from the pseudo-differential extension
\begin{equation}
    \begin{tikzcd}
        0 \ar[r] & \mathbb{K}(L^2(Y)) \ar[r] & \ensuremath{C^*} \Psi_{\textup{cl}}^0(Y) \ar[r, "\sigma^0"] & C(S^*Y) \ar[r] & 0
    \end{tikzcd}
\end{equation}
by \cite[Theorem 12.11]{Cuntz_Meyer_Rosenberg_2007}. Note that \(\sigma^0\colon  \Psi_{\textup{cl}}^0(Y)\to C^\infty(S^*Y)\) is a \(*\)-homomorphsim and hence extends to a map \(\sigma^0\colon \ensuremath{C^*}\Psi_{\textup{cl}}^0(Y)\to C(S^*Y)\). Also, \(\sigma^0\) has a left-inverse \(\Op\colon C^\infty(S^*Y)\to\Psi_{\textup{cl}}^0(Y)\) which is \(*\)-linear and multiplicative up to compacts. Hence, 
\begin{equation}
    \label{eq:pseudo_exension}
    C^\infty(S^*Y)\to \mathcal{Q}(L^2(Y)),~a\mapsto \Op(a) + \mathbb{K}(L^2(Y))
\end{equation}
extends by continuity to a Busby invariant \(C(S^*Y)\to \mathcal{Q}(L^2(Y))\) corresponding to the pseudo-differential extension.

\begin{lemma}
    \label{thm:cap_product_for_pseudo_projection}
    Let \(Y\) be a closed smooth manifold and \(E\to Y\) a smooth Hermitian vector bundle. Let \(P\in \Psi_{\textup{cl}}^0(Y;E)\) be a projection, then
    \begin{equation}
        \varphi_*([\sigma^0(P)]\cap [S^*Y]) = [P] \quad\text{in}\quad K^1(C(Y)),
    \end{equation}
    where \(\varphi\colon S^*Y\to Y\) the fiber map and the representation of \(C(Y)\) accompanying \(P\) is point-wise multiplication.
\end{lemma}

\begin{proof}

    To shorten notation, we write the class given by a Busby invariant \(\beta\colon A\to \mathcal{Q}(\mathcal{H})\) as \([a\mapsto \beta(a)]\), removing the need to name \(\beta\). We also allow for the expression \(a\mapsto \beta(a)\) to make sense only on a dense \(*\)-subalgebra of \(A\). We will show that
    \begin{equation}
        [\sigma^0(P)]\cap [S^*Y] = [a\mapsto P\Op(a)P + \mathbb{K}(L^2(Y;E))]
    \end{equation}
    in \(K^1(C(S^*Y))\), then result follows from that \(\varphi^*(a)\) is the symbol of multiplication by \(a\in C(Y)\).

    Let \(e\in M_N(C(Y))\) be a projection such that \(eC(Y)^N=C(Y;E)\). Since \(E\) is a smooth vector bundle we can choose \(e\in M_N(C^\infty(Y))\). We have that \(e\in \Psi_{\textup{cl}}^0(Y;\mathbb{C}^N)\) and as bounded operator on \(L^2(Y)^N\) it has range \(L^2(Y;E)\). Now we can extend \(P\) from \(L^2(Y;E)\) to \(L^2(Y)^N\) as \(Pe\). Note that \(Pe\) has symbol \(\sigma^0(P)\varphi^*(e)\) which is a projection in \(M_N(C^\infty(S^*Y))\). 

    Using the description of \([S^*Y]\) in \eqref{eq:pseudo_exension} and applying \Cref{thm:cap_prod_calculation} 
    \begin{equation}
        \begin{split}
        [\sigma^0(P)]\cap [S^*Y] 
        &= [\sigma^0(P)\varphi^*(e)]\cap [a\mapsto \Op(a) + \mathbb{K}(L^2(Y))] \\
        &= [a\mapsto \Op(\sigma^0(P)\varphi^*(e)a\sigma^0(P)\varphi^*(e)) + \mathbb{K}(L^2(Y)^N)] \\
        &= [a\mapsto \Op(\sigma^0(P)a\sigma^0(P)) + \mathbb{K}(L^2(Y;E))] \\
        &= [a\mapsto P\Op(a)P + \mathbb{K}(L^2(Y;E))]
    \end{split}
    \end{equation}
    where the last equality uses that \(\sigma^0(P\Op(a)P)=\sigma^0(P)a\sigma^0(P)\).
\end{proof}

Now for a classically elliptic differential operator \(D\) on a compact smooth manifold with boundary, combining \Cref{thm:boundary_data_kcycle_equality}, \Cref{thm:cap_product_for_pseudo_projection} and \Cref{thm:boundary_data_kcycle_is_welldefined_and_symbol_ktheory} we obtain that
\begin{equation}
    \label{eq:PkerDmax_equal_to_boundary_data}
    [P_{\Ker D_{\max}}] = \varphi_*\left([E_+(D)]\cap[S^*\partial \Omega]\right) \quad\text{in}\quad K^1(C(S^*\partial \Omega)).
\end{equation}
With \Cref{thm:boundary_of_diff_op_is_PkerDmax} we see that \(\partial [D]\) is the same class as in \eqref{eq:PkerDmax_equal_to_boundary_data} which concludes the proof of \Cref{thm:boundary_index_thm}.

\begin{remark}
    For a closed odd-dimensional \(\text{spin}^c\)~ manifold \(Z\) and \(\alpha\in\mathrm{GL}_N(C^\infty(Z))\), then  
    \begin{equation}
        \label{eq:atiyah_singer_index_theorem_thing}
        \braket{[\alpha],[Z]} = \int_Z \textup{ch}(\alpha)\wedge \textup{Td}(Z)
    \end{equation}
    \cite{Baum_van_Erp_2021} where \(\braket{\cdot,\cdot}\colon K_1(C(Z))\times K^1(C(Z))\to \mathbb{Z}\) is the index pairing, \(\textup{ch}(\alpha)\in H^{\text{odd}}_{\text{dR}} (Z;\mathbb{C})\) the Chern character of \(\alpha\) and \(\textup{Td}(Z)\in H^{\text{even}}_{\text{dR}} (Z;\mathbb{C})\) is the Todd class of \(Z\). Since \(\ind(P_{\Ker D_{\max}} \alpha P_{\Ker D_{\max}})=\braket{[\alpha], [P_{\Ker D_{\max}}]}\), we can use \eqref{eq:atiyah_singer_index_theorem_thing}, \eqref{eq:PkerDmax_equal_to_boundary_data} and functoriality to obtain \Cref{thm:intro_maximal_kernel_index_theorem}.
\end{remark}

\bibliographystyle{abbrv}
\bibliography{refs.bib}

\begin{thebibliography}{10}

\bibitem{Atiyah_1970}
M.~F. Atiyah.
\newblock Global theory of elliptic operators.
\newblock In {\em Proc. {{Internat}}. {{Conf}}. on {{Functional Analysis}} and
  {{Related Topics}} ({{Tokyo}}, 1969)}, pages 21--30. {Univ. Tokyo Press,
  Tokyo}, 1970.

\bibitem{Atiyah_Bott_1964}
M.~F. Atiyah and R.~Bott.
\newblock The index problem for manifolds with boundary.
\newblock In {\em Differential {{Analysis}}, {{Bombay Colloq}}., 1964}, pages
  175--186. {Oxford Univ. Press, London}, 1964.

\bibitem{Atiyah_Patodi_Singer_1975b}
M.~F. Atiyah, V.~K. Patodi, and I.~M. Singer.
\newblock Spectral asymmetry and {{Riemannian Geometry}}. {{I}}.
\newblock {\em Mathematical Proceedings of the Cambridge Philosophical
  Society}, 77(1):43--69, 1975.

\bibitem{Atiyah_Patodi_Singer_1975c}
M.~F. Atiyah, V.~K. Patodi, and I.~M. Singer.
\newblock Spectral asymmetry and {{Riemannian}} geometry. {{II}}.
\newblock {\em Mathematical Proceedings of the Cambridge Philosophical
  Society}, 78(3):405--432, 1975.

\bibitem{Atiyah_Patodi_Singer_1976a}
M.~F. Atiyah, V.~K. Patodi, and I.~M. Singer.
\newblock Spectral asymmetry and {{Riemannian}} geometry. {{III}}.
\newblock {\em Mathematical Proceedings of the Cambridge Philosophical
  Society}, 79(1):71--99, 1976.

\bibitem{Atiyah_Singer_1963}
M.~F. Atiyah and I.~M. Singer.
\newblock The index of elliptic operators on compact manifolds.
\newblock {\em Bulletin of the American Mathematical Society}, 69(3):422--433,
  1963.

\bibitem{Baaj_Julg_1983}
S.~Baaj and P.~Julg.
\newblock Th{\'e}orie bivariante de {{Kasparov}} et op{\'e}rateurs non
  born{\'e}s dans les {{C}}{\textbackslash}sp{\textbackslash}ast -modules
  hilbertiens.
\newblock {\em Comptes Rendus des S{\'e}ances de l'Acad{\'e}mie des Sciences.
  S{\'e}rie I. Math{\'e}matique}, 296(21):875--878, 1983.

\bibitem{Ballmann_Bar_2012}
W.~Ballmann and C.~B{\"a}r.
\newblock Boundary value problems for elliptic differential operators of first
  order.
\newblock {\em Surveys in Differential Geometry}, 17(1):1--78, 2012.

\bibitem{Bandara_2022}
L.~Bandara.
\newblock The {{Relative Index Theorem}} for {{General First-Order Elliptic
  Operators}}.
\newblock {\em The Journal of Geometric Analysis}, 33(1):10, 2022.

\bibitem{Bandara_Goffeng_Saratchandran_2023}
L.~Bandara, M.~Goffeng, and H.~Saratchandran.
\newblock Realisations of elliptic operators on compact manifolds with
  boundary.
\newblock {\em Advances in Mathematics}, 420, 2023.

\bibitem{Baum_Douglas_1982}
P.~Baum and R.~G. Douglas.
\newblock K-homology and index theory.
\newblock In {\em Operator Algebras and Applications, {{Part}} 1 ({{Kingston}},
  {{Ont}}., 1980)}, volume~38 of {\em Proc. {{Sympos}}. {{Pure Math}}.}, pages
  117--173. {Amer. Math. Soc., Providence, R.I.}, 1982.

\bibitem{Baum_Douglas_1982a}
P.~Baum and R.~G. Douglas.
\newblock Toeplitz {{Operators}} and {{Poincar\'e Duality}}.
\newblock In I.~Gohberg, editor, {\em Toeplitz {{Centennial}}: {{Toeplitz
  Memorial Conference}} in {{Operator Theory}}, {{Dedicated}} to the 100th
  {{Anniversary}} of the {{Birth}} of {{Otto Toeplitz}}, {{Tel Aviv}}, {{May}}
  11\textendash 15, 1981}, Operator {{Theory}}: {{Advances}} and
  {{Applications}}, pages 137--166. {Birkh\"auser}, {Basel}, 1982.

\bibitem{Baum_Douglas_Taylor_1989}
P.~Baum, R.~G. Douglas, and M.~E. Taylor.
\newblock Cycles and relative cycles in analytic {{K-homology}}.
\newblock {\em Journal of Differential Geometry}, 30(3):761--804, 1989.

\bibitem{Baum_van_Erp_2014}
P.~F. Baum and E.~{van Erp}.
\newblock K-homology and index theory on contact manifolds.
\newblock {\em Acta Mathematica}, 213(1):1--48, 2014.

\bibitem{Baum_van_Erp_2016}
P.~F. Baum and E.~{van Erp}.
\newblock {\(K\)}-homology and {{Fredholm}} operators {{II}}: Elliptic
  operators.
\newblock {\em Pure and Applied Mathematics Quarterly}, 12(2):225--241, 2016.

\bibitem{Baum_van_Erp_2018}
P.~F. Baum and E.~{van Erp}.
\newblock {\(K\)}-homology and {{Fredholm}} operators {{I}}: {{Dirac}}
  operators.
\newblock {\em Journal of Geometry and Physics}, 134:101--118, 2018.

\bibitem{Baum_van_Erp_2021}
P.~F. Baum and E.~{van Erp}.
\newblock The {{Index Theorem}} for {{Toeplitz Operators}} as a {{Corollary}}
  of {{Bott Periodicity}}.
\newblock {\em The Quarterly Journal of Mathematics}, 72(1-2):547--569, 2021.

\bibitem{Blackadar_1986}
B.~Blackadar.
\newblock {\em {\(K\)}-{{Theory}} for {{Operator Algebras}}}, volume~5 of {\em
  Mathematical {{Sciences Research Institute Publications}}}.
\newblock {Springer}, {New York, NY}, 1986.

\bibitem{Boutet_de_Monvel_1971}
L.~{Boutet de Monvel}.
\newblock Boundary problems for pseudo-differential operators.
\newblock {\em Acta Mathematica}, 126(none):11--51, 1971.

\bibitem{Brezis_2011}
H.~Brezis.
\newblock {\em Functional {{Analysis}}, {{Sobolev Spaces}} and {{Partial
  Differential Equations}}}.
\newblock {Springer}, {New York, NY}, 2011.

\bibitem{Bunke_1995}
U.~Bunke.
\newblock A {{{\(K\)}-theoretic}} relative index theorem and {{Callias-type
  Dirac}} operators.
\newblock {\em Mathematische Annalen}, 303(1):241--279, 1995.

\bibitem{Connes_1994}
A.~Connes.
\newblock {\em Noncommutative Geometry}.
\newblock {Academic Press}, {San Diego}, 1994.

\bibitem{Connes_Moscovici_1995}
A.~Connes and H.~Moscovici.
\newblock The local index formula in noncommutative geometry.
\newblock {\em Geometric \& Functional Analysis GAFA}, 5(2):174--243, Mar.
  1995.

\bibitem{Cuntz_Meyer_Rosenberg_2007}
J.~Cuntz, R.~Meyer, and J.~M. Rosenberg.
\newblock {\em Topological and {{Bivariant {\(K\)}-Theory}}}, volume~36 of {\em
  Oberwolfach {{Seminars}}}.
\newblock {Birkh\"auser}, {Basel}, 2007.

\bibitem{Dave_Haller_2020}
S.~Dave and S.~Haller.
\newblock The {{Heat Asymptotics}} on {{Filtered Manifolds}}.
\newblock {\em The Journal of Geometric Analysis}, 30(1):337--389, 2020.

\bibitem{Dave_Haller_2022}
S.~Dave and S.~Haller.
\newblock Graded hypoellipticity of {{BGG}} sequences.
\newblock {\em Annals of Global Analysis and Geometry}, 62(4):721--789, 2022.

\bibitem{Deeley_Goffeng_Mesland_2018}
R.~J. Deeley, M.~Goffeng, and B.~Mesland.
\newblock The bordism group of unbounded {{\(KK\)}-cycles}.
\newblock {\em Journal of Topology and Analysis}, 10(2):355--400, 2018.

\bibitem{Forsyth_Goffeng_Mesland_Rennie_2019}
I.~Forsyth, M.~Goffeng, B.~Mesland, and A.~Rennie.
\newblock Boundaries, spectral triples and {{{\(K\)}-homology}}.
\newblock {\em Journal of Noncommutative Geometry}, 13(2):407--472, 2019.

\bibitem{Gerontogiannis_2022}
D.~Gerontogiannis.
\newblock On finitely summable {{Fredholm}} modules from {{Smale}} spaces.
\newblock {\em Transactions of the American Mathematical Society},
  375(12):8885--8944, 2022.

\bibitem{Goffeng_Mesland_2015}
M.~Goffeng and B.~Mesland.
\newblock Spectral triples and finite summability on {{Cuntz-Krieger}}
  algebras.
\newblock {\em Documenta Mathematica}, 20:89--170, 2015.

\bibitem{Grensing_2012}
M.~Grensing.
\newblock Universal cycles and homological invariants of locally convex
  algebras.
\newblock {\em Journal of Functional Analysis}, 263(8):2170--2204, 2012.

\bibitem{Gromov_Lawson_1983}
M.~Gromov and H.~B. Lawson.
\newblock Positive scalar curvature and the {{Dirac}} operator on complete
  {{Riemannian}} manifolds.
\newblock {\em Publications Math\'ematiques de l'I.H.\'E.S.}, 58:83--196, 1983.

\bibitem{Grubb_1968}
G.~Grubb.
\newblock A characterization of the non local boundary value problems
  associated with an elliptic operator.
\newblock {\em Annali della Scuola Normale Superiore di Pisa - Scienze Fisiche
  e Matematiche}, 22(3):425--513, 1968.

\bibitem{Grubb_1996}
G.~Grubb.
\newblock {\em Functional {{Calculus}} of {{Pseudodifferential Boundary
  Problems}}}, volume~65 of {\em Progress in {{Mathematics}}}.
\newblock {Birkh\"auser}, {Boston, MA}, 1996.

\bibitem{Haskell_Wahl_2009}
P.~Haskell and C.~Wahl.
\newblock {{\(K\)}}-{{Homology Classes}} of {{Dirac Operators}} on {{Smooth
  Subsets}} of {{Singular Spaces}}.
\newblock {\em Rocky Mountain Journal of Mathematics}, 39(4):1245--1265, 2009.

\bibitem{Higson_Roe_2000}
N.~Higson and J.~Roe.
\newblock {\em Analytic {{{\(K\)}-homology}}}.
\newblock Oxford {{Mathematical Monographs}}. {Oxford University Press,
  Oxford}, 2000.

\bibitem{Hilsum_2010}
M.~Hilsum.
\newblock Bordism invariance in {{{\(KK\)}-theory}}.
\newblock {\em Mathematica Scandinavica}, 107(1):73--89, 2010.

\bibitem{Hormander_2007}
L.~H{\"o}rmander.
\newblock {\em The {{Analysis}} of {{Linear Partial Differential Operators
  III}}: {{Pseudo-Differential Operators}}}.
\newblock Classics in {{Mathematics}}. {Springer}, {Berlin, Heidelberg}, 2007.

\bibitem{Jensen_Thomsen_1991}
K.~K. Jensen and K.~Thomsen.
\newblock {\em Elements of {{{\(KK\)}-Theory}}}.
\newblock {Birkh\"auser}, {Boston, MA}, 1991.

\bibitem{Kaplansky_Berberian_1955}
I.~Kaplansky and S.~K. Berberian.
\newblock {\em Rings of Operators}.
\newblock Mathematics {{337A}}. {University of Chicago, Deptartment of
  Mathematics}, 1955.

\bibitem{Kasparov_1980}
G.~G. Kasparov.
\newblock The operator {{{\(K\)}-functor}} and extensions of
  {{\ensuremath{C^*}}}-algebras.
\newblock {\em Izvestiya Akademii Nauk SSSR. Seriya Matematicheskaya},
  44(3):571--636, 719, 1980.

\bibitem{Kordyukov_1991}
{\relax Yu}.~A. Kordyukov.
\newblock {\(L^p\)}-{{Theory}} of elliptic differential operators on manifolds
  of bounded geometry.
\newblock {\em Acta Applicandae Mathematica}, 23(3):223--260, 1991.

\bibitem{Lawson_Michelsohn_1989}
H.~B. Lawson, Jr. and M.-L. Michelsohn.
\newblock {\em Spin Geometry}, volume~38 of {\em Princeton {{Mathematical
  Series}}}.
\newblock {Princeton University Press, Princeton, NJ}, 1989.

\bibitem{Lions_Magenes_1972}
J.~L. Lions and E.~Magenes.
\newblock {\em Non-{{Homogeneous Boundary Value Problems}} and
  {{Applications}}}.
\newblock {Springer}, {Berlin, Heidelberg}, 1972.

\bibitem{McLean_2000}
W.~McLean.
\newblock {\em Strongly Elliptic Systems and Boundary Integral Equations}.
\newblock Cambridge University Press, Cambridge, 2000.

\bibitem{Roe_1993}
J.~Roe.
\newblock {\em Coarse Cohomology and Index Theory on Complete {{Riemannian}}
  Manifolds}, volume 104 of {\em Memoirs of the {{American Mathematical
  Society}}}.
\newblock {American Mathematical Society}, 1993.

\bibitem{Seeley_1971}
R.~Seeley.
\newblock Norms and {{Domains}} of the {{Complex Powers {\(A_B^z\)}}}.
\newblock {\em American Journal of Mathematics}, 93(2):299--309, 1971.

\bibitem{Seeley_1966}
R.~T. Seeley.
\newblock Singular {{Integrals}} and {{Boundary Value Problems}}.
\newblock {\em American Journal of Mathematics}, 88(4):781--809, 1966.

\bibitem{Shubin_2001}
M.~A. Shubin.
\newblock {\em Pseudodifferential {{Operators}} and {{Spectral Theory}}}.
\newblock {Springer}, {Berlin, Heidelberg}, 2001.

\bibitem{Taylor_2011}
M.~E. Taylor.
\newblock {\em Partial {{Differential Equations I}}: {{Basic Theory}}}, volume
  115 of {\em Applied {{Mathematical Sciences}}}.
\newblock {Springer}, {New York}, second edition, 2011.

\bibitem{van_Erp_Yuncken_2019}
E.~{van Erp} and R.~Yuncken.
\newblock A groupoid approach to pseudodifferential calculi.
\newblock {\em Journal f\"ur die reine und angewandte Mathematik (Crelles
  Journal)}, 2019(756):151--182, 2019.

\bibitem{Wahl_2009}
C.~Wahl.
\newblock Homological index formulas for elliptic operators over
  {{\ensuremath{C^*}}}-algebras.
\newblock {\em New York Journal of Mathematics}, 15:319--351, 2009.

\bibitem{Weidmann_1980}
J.~Weidmann.
\newblock {\em Linear {{Operators}} in {{Hilbert Spaces}}}, volume~68 of {\em
  Graduate {{Texts}} in {{Mathematics}}}.
\newblock Springer, New York, NY, 1980.

\end{thebibliography}

\end{document}